\documentclass[10pt,reqno]{amsart}
\usepackage{amsmath}
\usepackage{amssymb}
\usepackage{amsthm}
\usepackage{stackengine}
\usepackage{scalerel}
\usepackage{graphicx}
\usepackage[table]{xcolor}
\usepackage[inline]{enumitem}
\usepackage{hyperref}
\usepackage{cleveref}
\usepackage[normalem]{ulem}

\setlist[enumerate]{label={\upshape(\roman*)}}

\newcommand{\bD}{\begin{definition}\label}
\newcommand{\eD}{\end{definition}}

\newtheorem*{theorem*}{Theorem}
\newtheorem{theorem}{Theorem}[section]
\newtheorem{lemma}[theorem]{Lemma}
\newtheorem{corollary}[theorem]{Corollary}
\newtheorem{proposition}[theorem]{Proposition}

\theoremstyle{remark}
\newtheorem{definition}[theorem]{Definition}
\newtheorem{example}[theorem]{Example}
\newtheorem{question}{Question}

\newtheorem{remark}[theorem]{Remark}

\newcommand{\Sc}{\mathcal{S}}
\usepackage{amsrefs}
\theoremstyle{remark}

\DeclareMathOperator{\diag}{diag}
\DeclareMathOperator{\area}{Area}

\newcommand{\spans}{\operatorname{span}}
\newcommand{\range}{\operatorname{ran}}

\title[Singly generated Selfadjoint-Ideal operator semigroups]{\textbf{Singly generated Selfadjoint-Ideal operator semigroups: \\spectral density of the generator and simplicity}}

\begin{document}
\author{Sasmita Patnaik*}
\address{Department of Mathematics and Statistics, Indian Institute of Technology, Kanpur 208016 (INDIA)}
\email{sasmita@iitk.ac.in}
\author{Sanehlata}
\address{Department of Mathematics and Statistics, Indian Institute of Technology, Kanpur 208016 (INDIA)}
\email{snehlata@iitk.ac.in}
\author{Gary Weiss**}
\address{Department of Mathematics, University of Cincinnati, Cincinnati, OH 45221-0025 (USA)}
\email{gary.weiss@uc.edu}


\maketitle	
	
\begin{abstract}
This extends our new study of the automatic selfadjoint ideal property for $B(\mathcal H)$-operator semigroups introduced to us by Heydar Radjavi (SI semigroups for short). Our investigation here of singly generated SI semigroups led to unexpected algebraic and analytic phenomena on the simplicity of SI semigroups and on the spectral density of their generators.
In particular: the SI property yields for a hyponormal operator, zero planar area measure of its approximate point spectrum; the same for the essential spectrum of an essentially normal operator; and that SI semigroups generated by unilateral weighted shifts with periodic nonzero weights are simple. We also characterized the simplicity of the SI semigroups generated by certain commuting classes of normal operators.
\end{abstract}		

	
\section{Introduction}

In \cite{PW21} we began our investigation of a question posed to us by Heydar Radjavi in a private communication (2015):  Which multiplicative semigroups
in $B(\mathcal H)$ have all their multiplicative ideals (that is, semigroup-ideals) automatically
selfadjoint \cite[Definitions 1.1-1.3]{PW21}. We call these semigroups selfadjoint-ideal semigroups (SI semigroups for short). He pointed out, for instance, that in multiplicative semigroups $B(\mathcal H)$ and $F(\mathcal H)$, all multiplicative ideals are automatically selfadjoint. 
We found this SI property interesting because it turned out to be a unitary invariant for semigroups of $B(\mathcal H)$, which invariant we believe is new; and hence a useful tool in distinguishing between them up to unitary equivalence, and sometimes in determining their simplicity (i.e., whether or not they have no proper multiplicative ideals), a subject of considerable interest in semigroup theory.

Herein all ideals of semigroups are meant to be two-sided, and $B(\mathcal H)$ and $F(\mathcal H)$ respectively are regarded as the multiplicative semigroups of bounded linear operators and finite rank operators on a finite or infinite dimensional complex Hilbert space $\mathcal H$. 
In the study of semigroups, the possibility of general characterizations (or any general structure) of selfadjoint semigroups of $B(\mathcal H)$ (i.e., semigroups closed under adjoints, see Definition \ref{D1}) seems unexplored which complicates our investigation of SI semigroups even at the basic level of singly generated semigroups. Hence our main focus on these here and in \cite{PW21}.

An SI semigroup is a multiplicative semigroup $\mathcal S$ all of whose ideals are selfadjoint, or equivalently, in which for each $A$ in the semigroup $\mathcal S$, the bilinear operator equation $A^* = XAY$ is solvable for some $X, Y \in \mathcal S \cup \{I\}$ \cite[Lemma 1.9]{PW21}. 
Here we investigate further characterizations of SI for several special fundamental classes of semigroups and consequences of their possessing  the SI property. To our knowledge, the study of this bilinear operator equation in terms of the existence of solutions in a multiplicative semigroup in $B(\mathcal H)$ is new, 
as well as further potential for finding new simple semigroups from our techniques used to study the SI property. 

Our study began in \cite{PW21} with characterizations of SI semigroups inside certain classes of singly generated selfadjoint semigroups of $B(\mathcal H)$-operators. Our main focus turned out to be on singly generated selfadjoint semigroups $\Sc(T, T^*)$ generated by $T \in B(\mathcal H)$ (all finite products of $T$ and $T^*$). (Herein a singly generated selfadjoint semigroup we mean to be the semigroup generated by $T$ and $T^*$, hence the alternate name for $\Sc(T, T^*)$ ``singly generated selfadjoint semigroup" despite the double generators.)
Our aim there was and in part here is twofold: to study the impact of the SI property of $\Sc(T, T^*)$ on the structure of and constraints on special important classes of $T$ and to find characterizations (in terms of properties of $T$) of simple and non-simple ones. At first the investigation of which $\Sc(T, T^*)$ are SI semigroups for an arbitrary $T$ seemed intractable. So in \cite{PW21} we first considered the class of normal operators and, among the non-normals, the class of rank-one operators. For these two classes of operators, the SI property implied simplicity of $\Sc(T, T^*)$ in most cases and non-simplicity in rare cases. (See \cite[Section 3 end]{PW21} for a summary of results.)

In this paper, we continue the study of SI semigroups by focusing on singly generated selfadjoint semigroups $\Sc(T, T^*)$ generated by operators beyond our work on normal operators and rank-one operators $T$. More specifically, among the non-normal operators we study the special classes of operators: unilateral weighted shifts, hyponormal operators (which include subnormal operators), and essentially normal operators. (We say that $T$ is hyponormal if $T^*T-TT^* \geq 0$; an operator $T$ on a Hilbert space $\mathcal H$ is subnormal if there is a Hilbert space $\mathcal K$ containing $\mathcal H$ and a normal operator $S$ on $\mathcal K$ such that the restriction of $S$ to $\mathcal H$ is equal to $T$); and $T$ is essentially normal if the image of $T$ in the Calkin algebra $B(\mathcal H)/K(\mathcal H)$ is normal.) 
Our study led to interesting algebraic and analytic impacts of the SI property for $\Sc(T, T^*)$ on the simplicity of these $\Sc(T, T^*)$ semigroups and the spectral density of its generator $T$. And we also study the SI semigroups generated by a set of commuting normals as part of our study of the SI property of semigroups generated by more than one operator. Beyond sets of commuting normals, the study of the SI property of semigroups generated by arbitrary sets (possibly non-normals) remains open for us.  In the next four paragraphs, we summarize the work done in this paper. 

In the case of singly generated semigroup $\Sc(T, T^*)$ generated by a nonselfadjoint normal operator $T$, in \cite[Theorem 2.1]{PW21} we proved that simplicity is equivalent to the SI property. In this paper, we investigate the SI property of semigroups generated by commuting normals, by first determining necessary and sufficient conditions for the semigroup to be simple (Corollary \ref{operator simplicity}). In this case, semigroups are automatically SI semigroups because simple semigroups are trivially SI. In an attempt to characterize the nonsimple SI semigroups generated by a set of commuting normals, we only manage to determine necessary and sufficient conditions for a nonsimple SI semigroup when generated by a set of two normals (Corollary \ref{TSI}). We could not extend our techniques even to the 3 generator case, and so it remains to be addressed.  

In our investigation of the SI property of $\Sc(T, T^*)$ generated by unilateral weighted shifts with weight sequences $\{\alpha_n\}$, we considered two classes of weight sequences: weight sequences $\{\alpha_n\}$ with a zero gap (that is, for some $i \geq 1$ one has $\alpha_i \neq0$ and $\alpha_{i+1}= 0$) and weight sequences with no zero gap. For the class of weighted shifts with a zero gap, we determined equivalent conditions for $\Sc(T, T^*)$ to be SI (Theorem \ref{theorem 3.1}), and in this case $\Sc(T, T^*)$ is always nonsimple (Corollary \ref{nonsimple SI semigroups}). For the class of weighted shifts with no zero gap, we obtained a necessary condition (which is not sufficient, see Example \ref{p-example}) for the SI property for $\Sc(T, T^*)$ (Theorem \ref{THEOREM}). Nevertheless, we were able to obtain a necessary and sufficient condition for $\Sc(T, T^*)$ to be SI when generated by any $T$ from two particular subclasses of weighted shift operators from among those that have no zero gap, that is, $\{\alpha_j\}=0_N\oplus \{\alpha_j\}_{j>N}$ where $\alpha_j\ne 0$ for $j>N\ge 0$. Those classes are: those weighted shifts whose nonzero weights $\{\alpha_j\}_{j>N}$ have periodic absolute value sequence ($\{|\alpha_j|\}_{j>N}$); and those weighted shifts whose nonzero weights $\{\alpha_j\}_{j>N}$ have eventually constant absolute value sequence ($\{|\alpha_j|\}_{j>N}$) (Theorem~\ref{periodic SI characterization}, Corollary \ref{eventually-constant}).

Historically, among the non-normal operators, there has been a continuing interest in the study of hyponormal operators and essentially normal operators in terms of their spectral density (i.e., the various kinds of ``thinness'' of their spectrum), for example in \cite{Put70}, \cite{Stamp62}, and \cite{Stamp65}. In particular, the topological nature of the spectrum has been important in distinguishing hyponormal from normal operators in terms of various kinds of spectral thinness of these operators to force normality from hyponormality. For instance, if $T$ is hyponormal with a single limit point in its spectrum, then $T$ is normal \cite[Theorem 3]{Stamp62}; if the spectrum of $T$ is an arc, then $T$ is normal \cite[Theorem 4]{Stamp65}; if the planar area (i.e., the Lebesgue area which herein we simply call area) of its spectrum is zero, then $T$ is normal \cite[Corollary]{Put70}. And for a special class of hyponormals, namely, subnormal operators, the essential spectrum has been of interest as it provides a criterion to characterize those subnormal operators that are also essentially normal. For example, if $T$ is subnormal and the area of its essential spectrum is zero, then $T$ is essentially normal \cite[Corollary 31.15]{Conway}; and if the set of rational functions and the set of continuous functions acting on the essential spectrum of a subnormal operator are the same, then $T$ is essentially normal  \cite[Corollary 31.14]{Conway}. A natural question one might be interested in is: \textit{When does a subnormal or a hyponormal operator have any of the aforementioned spectral properties?} We provide some  partial  answers to this question 
 for singly generated SI semigroups $\Sc(T, T^*)$. More precisely, the SI property for $\Sc(T, T^*)$ generated by a hyponormal operator implies that the planar area of the approximate point spectrum of $T$ is zero (Lemma \ref{lemma6.1} and Remark \ref{R6.1}). As a consequence, whenever the boundary of the spectrum of hyponormal $T$ excludes at least one point of the unit circle, the SI property for $\Sc(T, T^*)$ implies normality of a hyponormal operator (Theorem \ref{hypo}). The SI property for $\Sc(T, T^*)$ generated by an essentially normal operator implies that the planar area of the essential spectrum of $T$ is zero (Corollary \ref{corollary6.4}); and for a subnormal operator, under the SI property for $\Sc(T, T^*)$, the essential normality of $T$ is equivalent to the planar area measure of the essential spectrum being zero (Corollary \ref{corollary7.3}). So in some cases, the SI property for $\Sc(T, T^*)$ implies some of the different topological constraints arising in the citations above. We prove our results using the notion of characters on singly generated unital $C^*$ algebras for certain classes of operators.
(The existence of characters on $C^*(T)$ under various spectral conditions for hyponormals was investigated by Bunce in \cite{JB70}.)

Analysis of the interconnections between the SI semigroup $\Sc(T, T^*)$ and the spectrum of $T$ also reveals an interesting connection to the singly generated unital $C^*$-algebra $C^*(T)$. Note that the elements of $\Sc(T, T^*)$ are words in $T$ and $T^*$ which along with the identity $I$ are the basic building blocks for unital $C^*(T)$. For $T$ normal, it is known that $C^*(T)$ contains nontrivial projections if and only if the spectrum of $T$ is disconnected (\cite[Theorem 2.1.13]{Murphy} and Corollary \ref{projection}). We prove in Corollary \ref{projection} that if $T$ is a non-invertible normal operator, then the SI property for $\Sc(T, T^*)$ implies that the spectrum of $T$ is disconnected and hence $C^*(T)$ contains nontrivial projections. 

Based on our investigations so far, we anticipate that for an arbitrary operator $T$, there may be deep connections between the SI property of $\Sc(T, T^*)$, the existence of characters on $C^*(T)$, and the topological (and analytical) nature of the spectrum of $T$. The core problem in this investigation is how to solve the bilinear operator equation mentioned earlier in a multiplicative semigroup in $B(\mathcal H)$. And we hope that our work stimulates further investigation in this subject which here and along with \cite{PW21} is at its early stage of development.

For the convenience of the reader we recall below the definitions and terminology from \cite{PW21}. 

\textbf{Terminology (Definitions \ref{D1}-\ref{D6})}

The terminology given in Definitions \ref{D1}-\ref{D3} is standard. The terminology in Definition \ref{D4} on the notion of selfadjoint-ideal semigroups and our focus in \cite{PW21}, we believe is new and due to Radjavi.

\bD{D1}
A semigroup $\Sc$ in $B(\mathcal H)$ is a subset closed under multiplication.
 A selfadjoint semigroup $\Sc$ is a semigroup also closed under adjoints,  i.e., $\Sc^* := \{T^* \mid T \in \Sc\} \subset \Sc$.
\eD
 
\bD{D3}
 An ideal $J$ of a semigroup $\mathcal S$ in $B(\mathcal H)$ is a subset of $\Sc$ closed under products of operators in $\Sc$ and $J$. That is, $XT, TY \in J$ for $T \in J$ and $X, Y \in \mathcal S$. And so also $XTY \in J$.
 \eD

   \bD{D4} A selfadjoint-ideal (SI) semigroup $\Sc$ in $B(\mathcal H)$ is a semigroup for which every ideal $J$ of $\Sc$ is closed under adjoints, i.e., $J^{*} := \{T^{*} \mid T \in J\} \subset J.$  
 \eD
Because this selfadjoint ideal property in Definition \ref{D4} concerns selfadjointness of all ideals in a semigroup, we call these semigroups  selfadjoint-ideal semigroups (SI semigroups for short).

\vspace{.2cm}

\textit{Semigroups generated by $\mathcal A \subset B(\mathcal H)$}
\bD{D5}
The semigroup generated by a set $\mathcal A \subset  B(\mathcal H)$, denoted by $\mathcal S(\mathcal A)$, is the intersection of all semigroups containing $\mathcal A.$ Also define $\mathcal A^*:= \{A^* | A \in \mathcal A\}$.
\eD
For short we denote by $\Sc(T)$ the semigroup generated by $\{T\}$ (called generated by $T$ for short).
It should be clear for the semigroup $\mathcal S(\mathcal A)$ that Definition \ref{D5} is equivalent to the semigroup consisting of all possible words of the form $A_1A_2\cdots A_k$ where $k \in \mathbb N$ and $A_i \in \mathcal A$ for each $1\leq i \leq k$.
 
 \bD{D6}
The selfadjoint semigroup generated by a set $\mathcal A \subset B(\mathcal H)$ denoted by $\Sc(\mathcal A \cup \mathcal A^{*})$ or $\Sc(\mathcal A,\mathcal A^{*})$, is the intersection of all selfadjoint semigroups containing  $\mathcal A \cup \mathcal A^{*}$.
Let $\Sc(T, T^*)$ denote for short $\Sc(\{T\}, \{T^*\})$ and call it the singly generated selfadjoint semigroup generated by $T$.
 \eD
 
It is clear that $\Sc(\mathcal A, \mathcal A^{*})$ is a selfadjoint semigroup. Moreover, it is clear that Definition \ref{D6} conforms to the meaning of $\Sc(\mathcal A \cup \mathcal A^*)$ in terms of words discussed above. That is, it consists of all words of the form $A_1A_2\cdots A_k$ where $k \in \mathbb N$ and $A_i \in \mathcal A \cup \mathcal A^{*}$ for each $1\leq i \leq k$.

The focus of this paper is the investigation of the singly generated SI semigroups $\Sc(T, T^*)$. So, we provide a description of the elements of $\Sc(T, T^*)$ here (see also \cite[Proposition 1.6]{PW21}).

For $T \in B(H)$, the semigroup $S(T, T^*)$ generated by the set $\{T, T^*\}$ is given by

\noindent $ S(T, T^*) = \{T^n, {T^*}^n, \Pi_{j=1}^{k}T^{n_j}{T^*}^{m_j},  (\Pi_{j=1}^{k}T^{n_j}{T^*}^{m_j})T^{n_{k+1}}, \Pi_{j=1}^{k}{T^*}^{m_j}T^{n_j}, (\Pi_{j=1}^{k}{T^*}^{m_j}T^{n_j}){T^*}^{m_{k+1}},$ where $n \ge 1,\,  k\ge1,\, n_j, m_j \ge 1\, \text{for}\, 1 \le j \le k, ~\text{and}~n_{k+1}, m_{k+1} \geq 1\}$.
The product $\Pi_{j= 1}^{k}$ in the semigroup list is meant to denote an ordered product.
Indeed, this follows directly from Definitions \ref{D5}-\ref{D6} and the accompanying word description by taking $\mathcal A = \{T\}$.

Alternatively $\Sc(T, T^*)$ consists of: words only in $T$, words only in $T^*$, words that begin and end in $T$, words that begin with $T$ and end with $T^*$, and words that begin with $T^*$ and end with $T$ and words that begin and end with $T^*$.
\color{black}

	\section{On simplicity of SI semigroups generated by commuting normals} 
	
In \cite[Section 3]{PW21} we obtained a complete characterization (i.e., necessary and sufficient conditions to possess the SI property) of semigroups $\Sc(T, T^*)$ generated by a rank-one operator $T$; and in some cases the SI property implied the simplicity of $\Sc(T, T^*)$. (A summary of the complete classification is provided in \cite[before Remark 3.21]{PW21}.) The various levels of difficulty and limited techniques at our disposal made us take a complicated approach to obtain the characterization for the SI semigroup $\Sc(T, T^*)$ in this simplest case of rank-ones. Further study of the SI semigroups generated by a finite rank operator beyond the rank-one operators will appear separately in a later paper.

Among many other results in \cite{PW21}, we characterized SI and simplicity for those semigroups $\Sc(T, T^*)$ generated by a single normal operator $T$. 
In this section we focus on SI and simplicity questions for semigroups of commuting normal operators and singly generated semigroups generated by unilateral weighted shifts. We exploit the GNS (Gelfand-Naimark-Segal) $C^*$-isometric isomorphism for commuting classes of normal operators to answer those questions.

Before we begin the investigation of SI semigroups generated by commuting normals  starting with those generated by two commuting normal operators and before moving on to arbitrary numbers of generators, 
we recall the singly generated SI semigroup characterization generated by a normal operator \cite[Theorem 2.1]{PW21}. And then investigate singly generated SI semigroups $\Sc(T, T^*)$ generated by an infinite rank weighted shift.

\begin{theorem*} \cite[Theorem 2.1]{PW21} For $T$ a normal nonselfadjoint operator, the following are equivalent.
\begin{enumerate}[label=(\roman*)]
\item $\Sc(T,T^*)$ is an SI semigroup.
\item $T$ is unitarily equivalent to $U \oplus 0$ (or $U$ when $\ker T = \{0\}$) with $U$ a unitary operator.
\item $\Sc(T,T^*)$ is a simple semigroup.
\end{enumerate}
\end{theorem*}

For $N_1$ and $N_2$ normal operators, we denote by $\Sc(N_1, N_1^*, N_2, N_2^*)$ the selfadjoint semigroup generated by $N_1$ and $N_2$. 
When commuting, i.e., $N_1N_2 = N_2N_1$, it follows from Putnam-Fuglede theorem (\cite[Problem 192]{Hal82}) that $N_iN_j^* = N_j^*N_i$ for $i=1,2$, and hence $\Sc(N_1, N_1^*, N_2, N_2^*)$ is an abelian semigroup. 
 And therefore, the following theorem reduces the questions on SI and simplicity of $\Sc(N_1, N_1^*, N_2, N_2^*)$ to questions on SI and simplicity of the corresponding  semigroup in the $C^*$-algebra of continuous functions which vanish at infinity on locally compact Hausdorff space 
$\Omega(\mathcal A)$, the set of characters (nonzero complex-valued homomorphisms on $\mathcal A$) on a nonzero abelian $C^*$-algebra $\mathcal A$.

\begin{theorem*} \cite[Theorem 2.1.10 (Gelfand)]{Murphy} If $\mathcal A$ is a nonzero abelian $C^*$-algebra, then the Gelfand representation 
$$\phi : \mathcal A \rightarrow C_0(\Omega(\mathcal A))$$
is an isometric $*$-isomorphism.
\end{theorem*}

We apply this theorem to the case when the $C^*$-algebra $\mathcal A$ is generated by two commuting normals, $N_1$ and $N_2$, i.e., $\mathcal A = C^*(N_1, N_2)$, in which case the functions $f = \phi (N_1)$ and $g = \phi(N_2)$ form the generators of the corresponding $C^*$-algebra $C_0(\Omega(\mathcal A))$. And, since the SI property and simplicity of $\Sc(N_1, N_1^*, N_2, N_2^*) \subset \mathcal A$ is preserved under the isometric $*$-isomorphism $\phi$, so it suffices to study the SI property and simplicity of $\Sc(f, g, \bar{f}, \bar{g}) \subset C_0(\Omega(\mathcal A))$.
In Proposition \ref{CN}--Corollary \ref{operator simplicity}, we determine necessary and sufficient conditions for the simplicity of the selfadjoint semigroup $\Sc(f, g, \bar{f}, \bar{g})$ and then for the semigroup $\Sc(f, g, \bar{f}, \bar{g})$ to be a nonsimple SI semigroup.

For brevity denote: $X:= \Omega(\mathcal A)$ and for $f \in C_0(X)$, denote by $S_f$ the support set of $f$ in $X$ for which obviously $S_f = S_{\bar{f}}$.

\begin{proposition}\label{CN}(2 generator function simplicity)
For $0 \ne f,g \in C_0(X)$, the semigroup $\Sc(f, g, \bar{f}, \bar{g})$ is simple if and only if 
$\chi_{S_f} = \chi_{S_g}$ (equivalently, $S_f = S_g$) and $\chi_{S_f} = fgW$ for some word 
$W \in \Sc(f, g, \bar{f}, \bar{g}) \cup \{1\}$. 
And the equivalence remains true after replacing $fg$ in the equation $\chi_{S_f} = fgW$ by the conjugate of either $f$ or $g$ or both.
\end{proposition}

\begin{proof}
Suppose $\Sc(f, g, \bar{f}, \bar{g})$ is simple. Then, in particular, the principal ideal generated by $fg \ne 0$ coincides with the entire semigroup, i.e., $(fg)_{\Sc(f, g, \bar{f}, \bar{g})} = \Sc(f, g, \bar{f}, \bar{g})$. 
That $fg \ne 0$ follows from the fact that $fg = 0$ implies 
$g \notin (f)_{\Sc(f, g, \bar{f}, \bar{g})} = \Sc(f, g, \bar{f}, \bar{g})$, against simplicity 
(since otherwise $g = fW$ for some 
$W \in \Sc(f, g, \bar{f}, \bar{g}) \cup \{1\}$, hence 
$0 \neq g^2 = gfW = 0$, contradiction).
Hence, 
\begin{equation}{\label{EE1}}
f = fg W'
\end{equation}

\begin{equation}{\label{EE2}}
g = fg W''
\end{equation}
for some words $W'$ and $W''$ in $\Sc(f, g, \bar{f}, \bar{g}) \cup \{1\}$. For $x \in S_f$, it follows from Equation (\ref{EE1}) that $g(x) \neq 0$ implying $S_f \subset S_g$. For $x \in S_g$, it follows from Equation (\ref{EE2}) that $f(x) \neq 0$ implying $S_g \subset S_f$. Hence $S_f = S_g$. Moreover, by substituting $g = fg W''$ in Equation (\ref{EE1}), one obtains $$f = f^{2}gW'W''.$$
For $x \in S_f$, $$1 = f(x)g(x)W'(x)W''(x)$$ and for $x \in Z_f$ (zero set of $f$),  $$0 = f(x)g(x)W'(x)W''(x).$$ Therefore, for $x \in S_f \cup Z_f = X$, $$\chi_{S_f} = fgW'W''$$
where $\chi_{S_f}$ denotes the characteristic function on $S_f$ and $W'W''$ is a word in $\Sc(f, g, \bar{f}, \bar{g}) \cup \{1\}$.

Conversely, suppose $S_f = S_g$ and $\chi_{S_f} = fgW$ for some word $W \in \Sc(f, g, \bar{f}, \bar{g})$. To prove simplicity of the semigroup $\Sc(f, g, \bar{f}, \bar{g})$, 
 it suffices to show that the principal ideal generated by any word in $\Sc(f, g, \bar{f}, \bar{g})$ coincides with the semigroup. In order to do so, it further suffices to show the generators $f, g, \bar{f}, \bar{g}$ are in every given principal ideal $(Y)_{\Sc(f, g, \bar{f}, \bar{g})}$ for $Y \in \Sc(f, g, \bar{f}, \bar{g})$. 
And since $f = f \chi_{S_f}, g = g \chi_{S_g}$, 
so also $\bar{f} = \bar{f}\chi_{S_f}, \bar{g} = \bar{g}\chi_{S_g}$, 
and because $S_f = S_g$ implies $\chi_{S_f} = \chi_{S_g}$, it suffices to show $\chi_{S_f} \in (Y)_{\Sc(f, g, \bar{f}, \bar{g})}$. 

Since $Y \in \Sc(f, g, \bar{f}, \bar{g})$, $Y = f^m\bar{f}^ng^k\bar{g}^l$ for some $m, n, k, l \geq 0$ not all zero (by definition of ``generated''). 
(Here we interpret exponent $0$ to mean that variable is absent, instead of $1$ which may not be in the semigroup.)
Clearly $\chi_{S_f} = \chi_{S_f}^r$ for all $r \geq 1$ and since $\chi_{S_f} = fgW$, one has $\chi_{S_f} = \bar{f}\bar{g}\bar{W}$ and finally 
$\chi_{S_f} = f^rg^r\bar{f}^r\bar{g}^rW^r\bar{W}^r$. Then choosing $r > max(n,m,k,l)$, 
one can factor out $Y =  f^m\bar{f}^ng^k\bar{g}^l$ from $\chi_{S_f} = f^rg^r\bar{f}^r\bar{g}^rW^r\bar{W}^r$, which places the latter in $ (Y)_{\Sc(f, g, \bar{f}, \bar{g})}$.
\end{proof}

In the above Proposition \ref{CN}, when $\Sc(f, g, \bar{f}, \bar{g})$ is simple, then $\chi_{S_f} = \chi_{S_g}$ and $\chi_{S_f} = fgW \in \Sc(f, g, \bar{f}, \bar{g}) \subset C_0(X)$, so is continuous. Hence one can define the projection $P:= \phi^{-1}(\chi_{S_f}) \in C^*(N_1, N_2)$, and obtain as a consequence of this theorem the following characterization of simple SI semigroups $\Sc(N_1, N_2, N_1^{*}, N_2^{*})$.

\begin{corollary}\label{operator}(2 generator operator simplicity)
For two commuting normal operators $N_1,N_2 \ne 0$, 
the 
semigroup $\Sc(N_1, N_2, N_1^{*}, N_2^{*})$ is simple if and only if the Gelfand transform supports of the operators, i.e., of their functions $\phi(N_1)$ and $\phi(N_2)$, are the same and the defined above projection $P= N_1N_2W$ for some word $W \in \Sc(N_1, N_2, N_1^{*}, N_2^{*})) \cup \{I\}$. And the equivalence remains true after replacing $N_1, N_2$ in the equation $P= N_1N_2W$ by the adjoint of either $N_1$ or $N_2$ or both.\end{corollary}

Proposition \ref{CN} generalizes nicely to semigroups with an arbitrary collection of generators.

\begin{theorem}\label{CN arbitrary generators}(Arbitrary generator function simplicity)
Let $\mathcal F \subset C_0(X)$ be a set of nonzero generators for the semigroup $\Sc(\mathcal F)$ (i.e., the set of all finite words in elements from $\mathcal F$).

The semigroup $\Sc(\mathcal F)$ is simple if and only if 
all functions in $\mathcal F$ have the same support $S$
 and for each finite subset of $\mathcal F, \{f_i\}^n_1$, one has $\chi_S = \Pi^n_1 f_iW$ for some word 
$W \in \Sc(\mathcal F) \cup \{1\}$.
\end{theorem}
\begin{proof}
Suppose $\Sc(\mathcal F)$ is simple. Then any two functions $f,g \in \mathcal F$ must have the same support, equivalently, all have the same support. 
This follows because by simplicity, the principal ideals $(f)_{\Sc(\mathcal F)} = \Sc(\mathcal F) = (g)_{\Sc(\mathcal F)}$, hence $f = gW$ and $g = fW'$ for some $W,W' \in \Sc(\mathcal F) \cup \{1\}$, which implies as in the proof of Propostion \ref{CN} that $f,g$ have the same zero set and hence the same support.
Since all functions in $\mathcal F$ have the same support, all finite products of them,  $\Pi^n_1 f_i$, have the same support, call it $S$. Then again by simplicity, 
$f_1 \in \Sc(\mathcal F) = ((\Pi^n_1 f_i)^2)_{\Sc(\mathcal F)}$, 
so $f_1 =  (\Pi^n_1 f_i)^2W$ for some $W \in \Sc(\mathcal F) \cup \{1\}$, and canceling one obtains 
$\chi_S =   (\Pi^n_1 f_i)(\Pi^n_2 f_i)W$, the desired condition.

Conversely, assume that all functions in $\mathcal F$ have the same support $S$
 and for all finite subsets of $\mathcal F, \{f_i\}^n_1$, one has $\chi_S = \Pi^n_1f_iW$ for some word 
$W \in \Sc(\mathcal F) \cup \{1\}$. To show simplicity it clearly suffices to show that every $h \in \Sc(\mathcal F)$ lies in every principal ideal $(f)_{\Sc(\mathcal F)}$ for every $f \in \Sc(\mathcal F)$. To show this principal ideal claim, for each $f \in\Sc(\mathcal F)$, one has $f = \Pi^n_1f_i^{k_i}$ for some finite subset $\{f_i\}^n_1 \subset \mathcal F$ and all $k_i > 0$. Choosing $r > \max \{k_i\}_1^n$ and using the hypothesis that $\chi_S = \Pi^n_1f_iW$ for some word 
$W \in \Sc(\mathcal F) \cup \{1\}$, one obtains 
\begin{equation*}\label{F}
\chi_S = \chi^r_S = f^r_1\cdots f^r_nW^r = \Pi^n_1f_i^{k_i}f^{r-k_1}_1\cdots f^{r-k_n}_nW^r
= ff^{r-k_1}_1\cdots f^{r-k_n}_nW^r
\end{equation*}
Hence, $\chi_S \in (f)_{\Sc(\mathcal F)}$. Moreover, since all functions in $\mathcal F$ have the same support $S$, the functions $h$ and $\chi_S$ share the same support. Multiplying both sides by $h$, we obtain
$h = h\chi_S = fhW' \in (f)_{\Sc(\mathcal F)}$, where $W' := f^{r-k_1}_1\cdots f^{r-k_n}_nW^r$, which is what was needed to be shown.
\end{proof}

As with Proposition \ref{CN}-Corollary \ref{operator}, Theorem \ref{CN arbitrary generators} has its normal operator application.

\begin{corollary}\label{operator simplicity}(Arbitrary generator operator simplicity)
For a commuting family of nonzero normal operators $\mathcal F$, 
the 
selfadjoint semigroup generated by $\mathcal F$, $\Sc(\mathcal F, \mathcal F^*)$, is simple if and only if the Gelfand transform supports of the operators, i.e., of their functions $\phi(N), N \in \mathcal F \cup \mathcal F^*$, are the same $S$ and for every finite subset of $\mathcal F \cup \mathcal F^*, \{N_i\}^n_1$,  the projection $P := \phi^{-1}\chi_S$ satisfies $P= \Pi_1^nN_iW$ for some word $W \in \Sc(\mathcal F, \mathcal F^*) \cup \{I\}$.\end{corollary}

This ends our characterization of simplicity for semigroups in terms of their generators. We next investigate the relationship between SI semigroups and their generators, and recall that $C_0(X)$ denotes the range of the Gelfand map $\phi$.
\begin{theorem}\label{TSI}(Doubly generated function nonsimple SI)
For $0 \ne f,g \in C_0(X)$, the semigroup $\Sc(f, g, \bar{f}, \bar{g})$ is SI and nonsimple if and only if $\bar{f} = fW$ and $\bar{g} = gW'$ for some $W$ and $W'$ words in $f, g, \bar{f}, \bar{g},1$ but if
$g$ appears in $W$ (or $\bar{g}$ appears in $W$ (i.e., it is not necessary to check both)), then $W'$ must be a word in $g, \bar{g}$ only; and similarly, if  $f$ appears in $W'$ (or $\bar{f}$ appears in $W'$ (i.e., it is not necessary to check both)), then $W$ must be a word in $f, \bar{f}$ only.  (Of course then, if both $W$ is $g$ free and $W'$ is $f$ free, there is nothing to check.)
\end{theorem}
\begin{proof}
Suppose the semigroup $\Sc(f, g, \bar{f}, \bar{g})$ is SI and nonsimple. Since $\Sc(f, g, \bar{f}, \bar{g})$ is SI, the principal ideals $(f)_{\Sc(f, g, \bar{f}, \bar{g})}$ and $(g)_{\Sc(f, g, \bar{f}, \bar{g})}$ are selfadjoint. Hence $$\bar{f} = fW \quad \text{and} \quad \bar{g} = gW',$$ for some words $W, W' \in \Sc(f, g, \bar{f}, \bar{g}) \cup \{1\}$. We next prove that if $g$ appears in $W$ then $W'$ is a word in $g, \bar{g}$ only; and if $f$ appears in $W'$ then $W$ is a word in $f, \bar{f}$ only. 
Suppose that $g$ appears in $W$ and $f$ appears in $W'$, then
\begin{equation}{\label{EE33}}
\bar{f} = fgW_1 \quad \text{and} \quad \bar{g} = fgW_2 \quad \text{for} \quad W_1, W_2 \in \Sc(f, g, \bar{f}, \bar{g}) \cup \{1\}.
\end{equation}
Then because $\bar{f} = fgW_1$ one has $S_g \supset S_f$ 
and from $ \bar{g} = fgW_2$ one has $S_g \subset S_f$,
one then obtains $S_f = S_g$. 
Then from (\ref{EE33}) one has 
$f = \bar{f}\bar{g}\bar{W_1} \quad \text{and} \quad g =\bar{f}\bar{g}\bar{W_2}$, 
and thus $\bar{f}f\bar{g}g = fgW_1\bar{f}\bar{g}\bar{W_1}fgW_2\bar{f}\bar{g}\bar{W_2} 
= (\bar{f}f\bar{g}g)^2W_1\bar{W_1}W_2\bar{W_2}$. 
And because from $S_f = S_g$ one has $S_f = S_{\bar{f}f\bar{g}g}$, it follows that
$\chi_{S_f} = fg\bar{f}\bar{g}W_1\bar{W_1}W_2\bar{W_2}$ with the product $W_1\bar{W_1}W_2\bar{W_2} \in \Sc(f, g, \bar{f}, \bar{g}) \cup \{1\}$.
Hence, by Proposition \ref{CN}, $\Sc(f, g, \bar{f}, \bar{g})$ becomes simple against nonsimplicity of $\Sc(f, g, \bar{f}, \bar{g})$. Therefore, if $g$ appears in $W$ (and also by a symmetric argument if $\bar{g}$ appears in $W$), then $W'$ is a word in $g, \bar{g}$ only. Similarly it follows that if $f$ appears in $W'$, then $W$ is a word in $f, \bar{f}$ only.

Conversely, suppose $\bar{f} = fW$ and $\bar{g} = gW'$ where $W$ and $W'$ are words in $f, g, \bar{f}, \bar{g}$ such that if $g$ appears in $W$, then $W'$ is a word in $g, \bar{g}$ only; and if $f$ appears in $W'$, then $W$ is a word in $f, \bar{f}$ only. We claim that $\Sc(f, g, \bar{f}, \bar{g})$ is not simple. Indeed, if it were simple, then by Proposition \ref{CN}, $S_f = S_g$ and $\chi_{S_f} = fgW''$ for some word $W'' \in \Sc(f, g, \bar{f}, \bar{g}) \cup \{1\}$. 
This implies that $\bar{f} = \bar{f}\chi_{S_f} = f\bar{f}gW''$ and $\bar{g} = \bar{g}\chi_{S_f} = gf\bar{g}W''$, contradicting the hypothesis that if $\bar{f} = fW$ and $\bar{g} = gW'$, then whenever $g$ appears in $W$, then $W'$ is a word in $g, \bar{g}$ only.

We next prove that $\Sc(f, g, \bar{f}, \bar{g})$ is SI. For this it clearly suffices to show that the principal ideal generated by any word is selfadjoint. Let $X \in \Sc(f, g, \bar{f}, \bar{g})$. Then $X = f^n\bar{f}^mg^k\bar{g}^l$ for some $n, m, k, l \geq 0$. 
 For $W, W'$ given in the hypothesis $\bar{f} = f W$ and $\bar{g} = gW'$ respectively, consider the word $W^n\overline{W}^mW'^k\overline{W'}^l \in \Sc(f, g, \bar{f}, \bar{g}) \cup \{1\}$. By multiplying this word by $X$, rearranging, and using $\bar{f}^j = f^jW^j$ for $j = n, m$ and $\bar{g}^i = g^iW'^i$ for $i = k, l$, one obtains $$W^n\overline{W}^mW'^k\overline{W'}^lX = f^nW^n\bar{f}^m\overline{W}^mg^kW'^k\bar{g}^l\overline{W'}^l = \bar{f}^nf^m\bar{g}^kg^l = X^*.$$
Therefore, $X^* \in (X)_{\Sc(f, g, \bar{f}, \bar{g})}$.
Since $X$ is an arbitrary word in $\Sc(f, g, \bar{f}, \bar{g})$, every principal ideal of $\Sc(f, g, \bar{f}, \bar{g})$ is selfadjoint, from which it follows easily that $\Sc(f, g, \bar{f}, \bar{g})$ itself is SI.
\end{proof}
\begin{corollary}\label{NCN}(Doubly generated normal operator nonsimple SI)
For two commuting normal operators $N_1,N_2\ne 0$, the semigroup $\Sc(N_1, N_2, N_1^{*}, N_2^{*})$ is SI and nonsimple if and only if $N_1^* = N_1W$ and $N_2^* = N_2W'$ for some $W$ and $W'$ words in $N_1, N_2, N_1^*, N_2^*$ such that if
 $N_2$ appears in $W$, then $W'$ is a word in $N_2, N_2^*$ only; and if $N_1$ appears in $W'$, then $W$ is a word in $N_1, N_1^*$ only. 
Alternatively this equivalence remains true replacing either or both of the operators by their adjoints. 
\end{corollary}
Interestingly we see here that $\Sc(N_1, N_2, N_1^{*}, N_2^{*})$ cannot be singly generated since otherwise it would be simple by \cite[Theorem 2.1]{PW21}.

The 3 generator case remains open and hence also the arbitrary generator case.

\section{SI semigroups $\Sc(T, T^*)$ for $T$ a unilateral weighted shift}
Unilateral weighted shifts are often considered as a litmus test by operator theorists whenever a new concept is introduced. We next investigate the SI property for $\Sc(T, T^*)$ for $T$ a unilateral weighted shift of infinite rank.
As pointed out in \cite[second last paragraph of Introduction]{PW21} that SI semigroups are mostly simple and rarely nonsimple, construction of nonsimple SI semigroups are, in general, difficult. Here we provide a class of examples of nonsimple SI semigroups in Corollary \ref{nonsimple SI semigroups}, i.e., the SI semigroups generated by weight sequences with zero gap are always nonsimple. These singly generated selfadjoint semigroups are subsemigroups of the class of all weighted shifts relative to a fixed basis. 

From here on we refer to unilateral weighted shifts simply as weighted shifts. First recall the definition of a  weighted shift $T$: for $\{e_n \mid n\geq 1\}$ an orthonormal basis of $\mathcal H$ and $\{\alpha_n\}$ a bounded sequence of complex scalars not all zero, the operator defined by $Te_n = \alpha_ne_{n+1}$ for $n \geq 1$ and extended by linearity is the  weighted shift with weight sequence $\{\alpha_n\}$.

 A few facts used in the later parts of this section are: If $T$ is the  weighted shift with weight sequence $\{\alpha_n\}$, then for $i\geq 1$ and $m\geq 1$,
\begin{equation}\label{W0}
Te_i = \alpha_ie_{i+1},\, {T^*}e_{i+1}=\bar{\alpha}_ie_i \text{ and } T^*e_1 = 0.
\end{equation}  
Equation (\ref{W0}) implies that for $k \geq 1$ and $m \geq 1$,
\begin{equation} \label{def1'}
T^me_k=(\Pi_{j=k}^{k+m-1}\alpha_j)e_{k+m}
\end{equation} and 
\begin{equation}\label{def2'}
{T^*}^me_k=
\begin{cases} 
0 &  \quad \text{ for } 1\leq k \leq m\\
(\Pi_{k-m}^{k-1}\bar{\alpha}_j)e_{k-m} & \quad \text{ for } k \geq m+1 \\ 
\end{cases}
\end{equation}
Furthermore, the matrix representation in the basis $\{e_n\}$ of each word in $T$ and $T^*$ has exactly one nonzero diagonal (that is, strictly lower,  strictly upper or main diagonal). Going forward in this paper, it will be clear from Proposition~\ref{multiplication} that if $A, B$ are any two diagonal matrices (that is, strictly lower, strictly upper or main diagonal), then their product $AB$ is a strictly upper, strictly lower or main diagonal matrix. (We make a note here about $\Sc(T, T^*)$ in the context of graded semigroups which is not related to our study of the SI property, but could be of independent interest. Specifically, our singly generated selfadjoint semigroup $\Sc(T, T^*)$ generated by a (unilateral) weighted shift $T$ forms a subsemigroup of the semigroup generated by the set of all weighted shift operators and it forms a \textit{strongly} $\mathbb Z$-graded semigroup \cite[Definition 2.9]{HM} where $\mathbb Z$ is the additive group; and by taking the set $X = \mathbb Z$ and the semigroup $\mathcal S = \Sc(T, T^*)$ in \cite[first paragraph, Section 2.3]{HM}, we observe that the set $\mathbb Z$ is an $\mathcal S$-biset, but not pointed.)

We recall that by \cite [Corollary 1.15]{PW21}, if $T$ is a power partial isometry, then $\Sc(T,T^*)$ is always SI. The converse need not be true. Indeed one can construct examples where $T$ is a partial isometry and $\Sc(T, T^*)$ is SI, but still $T$ is not a power partial isometry. For example, consider the $2 \times 2$ matrix with the first column $1/\sqrt 2$ and the second column zero. Then, by \cite[Lemma 3.14]{PW21}, the semigroup generated by this matrix is SI because the trace is a nonzero real number, but this matrix is not a power partial isometry by \cite[Proposition 3.6]{PW21} because the trace of this matrix is $1/\sqrt 2 \neq 0$ and not on the unit circle.
	It is proved in the following theorem that if we consider a certain class of  weighted shifts (that is, when their weight sequence has a zero gap as defined in the theorem below), then the converse holds. Note  an easy computation tells us that for $T$ a  weighted shift with weight sequence $\{\alpha_n\}$, $T$ is a power partial isometry if and only if $|\alpha_{n}| \in \{0,1\}$ for $n\geq 1$. Theorem \ref{theorem 3.1} below says that this equivalence is further equivalent to $\Sc(T, T^*)$ being SI for a restricted class of weight sequences. (The case when there are no zero gaps is more difficult in characterizing which possess the SI property, even in providing only necessary conditions
	which will be covered later in Theorem \ref{THEOREM}.)
	
But first for extensive use in the proof of the next Theorem \ref{theorem 3.1}, we prove the following proposition. We shall refrain from referencing it when we use it.

Relative to the basis $\{e_n\}$, for diagonal $D = \diag \{d_n\}$ and $T$ a weighted shift with weights $\{\alpha_i\}$, $T^k$ is the $k^{th}$ subdiagonal with weights $\beta_i$ given by Equations (\ref{def1'})-(\ref{def2'}). Then for $x = \sum x_ne_n \in \mathcal H$, one has 
$DT^kx = \sum d_{n+k}\beta_n x_n e_{n+k}$ 
and solving for $y$ in $DT^kx = T^ky = \sum \beta_n y_n e_{n+k}$ yields $y_n = d_{n+k} x_n$ putting $y = \sum y_n e_n \in \mathcal H$, That is, 

\begin{proposition}\label{DT^k, DT^{*k} range lemma}
For $D$ a diagonal matrix and $T$ a weighted shift: 
$\range DT^k \subset \range T^k$ and $\range DT^{*k} \subset \range T^{*k}$ for all $k \ge 1$.
\end{proposition}

\begin{theorem}\label{theorem 3.1}
	For $T$ a  weighted shift with weight sequence $\{\alpha_n\}$ with its weight sequence having a zero gap (that is, for some $i \geq 1$ one has $\alpha_i \neq0$ and $\alpha_{i+1}= 0$), then the following are equivalent.
	\begin{enumerate}[label=(\roman*)]
	\item $\Sc(T,T^*)$ is an SI semigroup.
	\item $|\alpha_{n}| \in \{0,1\}$ for $n\geq 1$.
	\item $T$ is a power partial isometry.
		\end{enumerate}
		Furthermore, these SI semigroups $\Sc(T, T^*)$  are always nonsimple (see Corollary \ref{nonsimple SI semigroups} below).	\end{theorem}
\begin{proof}
(i)$\Rightarrow$(ii). We first claim that facts useful in this proof that $\range T \not\subset  \range T^*$ and $\range T \not\supset \range T^*$. Indeed, by hypothesis, for some $i \geq 1$, $\alpha_i \neq0$ and $\alpha_{i+1} = 0$, so $e_{i+1} \in \range T$ but $e_{i+1} \not \in \range T^*$ (follows from Equation (\ref{W0})). This implies that $\range T \not\subset  \range T^*$. For the reverse non-inclusion, let $r$  be the smallest index such that $\alpha_r \neq 0$. If $r=1$, then $e_1 \in \range T^*$ but $e_1 \not \in \range T$. If $r>1$, then $\alpha_1 = \cdots = \alpha_{r-1} = 0$ and so $e_r \in \range T^*$, but $e_r \notin \range T$ (follows below from the case $m=1$ in Equation (\ref{W1})). Hence, $\range T^* \not\subset \range T$.

We note here some observations (derived from Equations (\ref{def1'})-(\ref{def2'}) and from the fact that $T^*T$ is the diagonal with weights $|\alpha_n|^2$) that will be used in proving that if $\Sc(T,T^*)$ is SI then $|\alpha_{n}| \in \{0,1\}$ for $n\geq 1$.  
	
	For $x \in \mathcal H$, $x = \sum^{\infty}_{j=1}a_je_j$  where $a_j$'s are the Fourier coefficients of $x$ with respect to the orthonormal basis $\{e_j\}$ and for $m\geq 1$, one has 
	
	\begin{equation}\label{W1}
	\range T^m=\{T^mx \mid x \in \mathcal H\} =\{\sum^{\infty}_{j=1}a_j\alpha_j\cdots\alpha_{m+j-1}e_{j+m} \mid x \in \mathcal H\} 
	\end{equation}
	\begin{equation}\label{W2}
	\range {T^*}^2 = \{{T^*}^2x \mid x\in \mathcal H\} = \{\sum^{\infty}_{j\geq 3}a_j\bar{\alpha}_{j-2}\bar{\alpha}_{j-1}e_{j-2} \mid x \in \mathcal H\}
	\end{equation}
	\begin{equation}\label{W3}
	\range T^*T^2 = \range (T^*T)T  = \{(T^*T)Tx \mid x \in \mathcal H\} =\{\sum^{\infty}_{j=1}a_j\alpha_j|\alpha_{j+1}|^2e_{j+1} \mid x \in \mathcal H\}.
	\end{equation}
	\begin{equation}\label{W4}
	\range (T^*T)^mT^*T^2 = \range (T^*T)^{m+1}T  = \{(T^*T)^{m+1}Tx \mid x \in \mathcal H\} =\{\sum^{\infty}_{j=1}a_j\alpha_{j}|\alpha_{j+1}|^{2(m+1)}e_{j+1} \mid x \in \mathcal H\}.
	\end{equation}

	Since $\Sc(T,T^*)$ is SI, the principal ideal $(T)_{\Sc(T,T^*)}$ is selfadjoint. Therefore, $T^*=XTY$ for some $X,Y \in \Sc(T,T^*) \cup\{I\}$ and $X,Y$ cannot both be the identity operator (since otherwise $T$ would be selfadjoint but $T$ being a shift, is clearly not selfadjoint).
	
	We next claim that $T^* = T^*X'TY'T^*$ for some $X',Y' \in \Sc(T,T^*) \cup\{I\}$.
	As we proved above, $\range T^* \not\subset \range T$, so from $T^*=XTY$ one has $X \ne I$ nor can  $X$  start with $T$, and hence must start with $T^*$. That is, $XTY={T^*}^mX'TY$ for some $m\geq 1$ and $X'\in \Sc(T,T^*) \cup\{I\}$. 
	We now claim that $m = 1$. Otherwise $m \ge 2$ and $T^* = XTY={T^*}^2{T^*}^{m-2}X'TY$ 
	(interpreting $T^{*0}$ to mean absence) which  implies that 
	$\range T^* \subset \range {T^*}^2$. But $\range T^* \not \subset \range {T^*}^2$ because by hypothesis, $\alpha_i \neq0$ and $\alpha_{i+1}= 0$, so $e_i \in \range T^*$ but $e_i \not \in \range {T^*}^2$ (via Equations (\ref{W0}),(\ref{W2})), a contradiction. Next, $XTY$ must end with $T^*$.  
	Indeed, if $XTY$ ends with $T$, then $T^* = XTY'T$ for some $Y'\in \Sc(T,T^*) \cup\{I\}$. Taking adjoints we obtain $T = T^*Y'^*T^*X^*$ which implies that $\range T \subset \range T^*$, contradicting $\range T \not\subset  \range T^*$ which we proved in the first paragraph of this proof.
	
	We next show that $XTY$ starts with $T^*$ then alternates between $T$ and $T^*$ and ends with $T^*$, that is, $XTY=(T^*T)^mT^*$ for some $m\geq 1$. 
	But first let $r$  be the smallest index for which $\alpha_r \neq 0$, 
	so for $r>1$, $\alpha_1 = \cdots = \alpha_{r-1} = 0$. 
	And note if $r=1$, then $e_1 \in \range T^*$ 
	but by Equation (\ref{W3}), $e_1 \not \in \range T^*T^2$.

	Now suppose otherwise that $XTY$ is a word in powers of $T$ and $T^*$, beginning with $T^*T$ and ending in $T^*$ (as proved just above), but with at least one higher power of $T$ or $T^*$ appearing.  There are clearly three possiblilities for the beginning terms of $XTY$ beginning with $T^*T$ and ending in $T^*$: 
	
	$T^* = XTY = T^*T^kY'$ for some $k \ge 2$ and $Y' \in \Sc(T,T^*)$ beginning and ending with $T^*$;
	
	$(T^*T)^mT^kY'$ for some $m \ge 1, k \ge 1$ for some $Y' \in \Sc(T,T^*)$ beginning and ending with $T^*$;
	
	or $(T^*T)^mT^{*k}Y'$ for some $m \ge 1, k \ge 1$ for some $Y' \in \Sc(T,T^*) \cup \{I\}$ but when $Y' \ne I$, then it begins with $T$ and ends with $T^*$.
	
	The first case fails since $e_r \in \range T^* \setminus \range T$ via the first paragraph of proof and since 
	$T^*T$ is diagonal,  has range $ (T^*T)T$ included in $\range T$, and together with $T^* = (T^*T)T^{k-1}Y'$ (with $k - 1 \ge 1$) implies $\range T^* \subset \range T$, 
	contradicting $e_r \in \range T^* \setminus \range T$.
	
	The second case fails because 
	for $(T^*T)^mT^kY'$, since $(T^*T)^m$ is diagonal, $(T^*T)^mT$ has range included in $\range T$ which implies
	$T^* = (T^*T)^mT^kY'$ has $\range T^* \subset \range T$, again a contradiction as just before.
	
	And the third case fails for $k \ge 2$ but passes for $k=1$ or leads naturally, by increasing $m$ and repeating the process to case 2 or 3 again, towards the concluding form $XTY=(T^*T)^mT^*$.
	
	To show the third case fails for $k \ge 2$, since $(T^*T)^m$ is diagonal, one has 
	$\range T^* = \range (T^*T)^mT^{*k}Y' \subset \range (T^*T)^mT^{*2} \subset \range T^{*2}$.
	Hence it suffices to show $\range T^* \not\subset \range T^{*2}$ to obtain a contradiction. 
	But now recall from the first paragraph of this proof that $\alpha_i \ne 0$ and $\alpha_{i+1} = 0$ implies 
	$e_i \in \range T^*$ but it is easy to see using Equation (\ref{W2}) that $e_i \notin  \range T^{*2}$.
	
	From $T^*=XTY=(T^*T)^mT^*$ for some $m\geq 1$, and since $T^*T= \diag(|\alpha_1|^2,|\alpha_2|^2,|\alpha_3|^2,\dots)$, 
	by right multiplying $T$ we get $T^*T=(T^*T)^{m+1}$. And because $T^*T$ is diagonal, by equating the diagonal entries in this equation one obtains $|\alpha_j| \in \{0,1\}$.
	
	(ii)$\Rightarrow$(iii) Since $|\alpha_n| \in \{0,1\}$, ${T^*}^kT^k$ is a projection by computation, and hence $T^k$ is a partial isometry for each $k\geq 1$, that is,  
	$T$ is a power partial isometry.
	
	(iii)$\Rightarrow$(i) By \cite[Corollary 1.15]{PW21}, $\Sc(T,T^*)$ is SI since $T$ is a power partial isometry.
\end{proof}

In the next proposition, we prove that all semigroups $\Sc(T,T^*)$ generated by weighted shifts with the weight sequence having a zero gap are always nonsimple. In particular, all the SI semigroups $\Sc(T,T^*)$ provided by the characterization in Theorem \ref{theorem 3.1} are nonsimple.
\begin{proposition}\label{nonsimple semigroups}
	The semigroups $\Sc(T,T^*)$ generated by weighted shifts $T$ with the weight sequence $\{\alpha_n\}$ that has the gap property ($\alpha_i\ne 0, \alpha_{i+1}=0$ for some $i\ge 1$) are nonsimple. 
\end{proposition}
\begin{proof}
	To prove the nonsimplicity of $\Sc(T,T^*)$, it suffices to find a proper ideal. In particular, we prove $(T^{i+1})_{\Sc(T,T^*)}$ is a proper ideal that does not contain $T$, where $i \ge 1$ is the gap spot. One can check that the gap ensures that the subspace $\mathcal{M}:=\spans\{e_1,e_2,\dots,e_{i+1}\}$ is a reducing subspace for $T$. Let $A:=T\mid\mathcal{M}$ and $S:=T\mid \mathcal{M}^{\perp}$. Then $A$ is a weighted shift matrix (w.r.t. the basis $\{e_1,e_2,\dots,e_{i+1}\}$) of size $i+1$, which is nilpotent of nilpotency degree $\le i+1$, and $S$ is a weighted shift on $\mathcal{M}^{\perp}$ ( w.r.t. the basis $\{e_n\}_{n\ge i+2}$) with weight sequence $\{\alpha_n\}_{n\ge i+2}$. Therefore $T^{i+1}=A^{i+1}\oplus S^{i+1}=0\oplus S^{i+1}$, where $0$ denotes the zero matrix of size $i+1$. Then for any $B\in (T^{i+1})_{\Sc(T,T^*)}$ one has $B=XT^{i+1}Y$ for some $X,Y\in \Sc(T,T^*)\cup\{I\}.$ Then the equation $B=XT^{i+1}Y$ can be rewritten w.r.t. $\mathcal{M}\oplus \mathcal{M}^{\perp}$ as $B=X_1\oplus X_2(0\oplus S^{i+1})Y_1\oplus Y_2=0\oplus X_2S^{i+1}Y_2$, where $X_1,Y_1\in \Sc(A,A^*)\cup\{I_1\}$ and $X_2,Y_2\in \Sc(S,S^*)\cup \{I_2\}$ and $I_1,I_2$ are indentity operators on $\mathcal{M},\mathcal{M}^{\perp}$ respectively. Hence $B\mid \mathcal{M}=0$ for every $B\in (T^{i+1})_{\Sc(T,T^*)}$, which implies that $T\notin(T^{i+1})_{\Sc(T,T^*)}$ because $Te_i=\alpha_ie_{i+1}\ne 0$. Therefore, $(T^{i+1})_{\Sc(T,T^*)}$ is a non-trivial ideal and so $\Sc(T,T^*)$ is a nonsimple semigroup.
\end{proof}
A special case of the above proposition is
\begin{corollary}\label{nonsimple SI semigroups}
	The SI semigroups $\Sc(T,T^*)$ provided by the characterization in Theorem \ref{theorem 3.1} are nonsimple.
\end{corollary}

Theorem \ref{theorem 3.1} provides a characterization of SI semigroup $\Sc(T,T^*)$ for those weighted shifts whose weight sequence has the gap property that $\alpha_i \neq 0$, 
$\alpha_{i+1}= 0$ for some $i \ge 1$. So to fully characterize the SI semigroups $\Sc(T, T^*)$ generated by arbitrary weighted shifts $T$, the only case remaining to investigate is the case when there are no gaps, that is, all $\alpha_j \neq 0$ for $j > N$ for some $N \ge 0$, i.e., $\{\alpha_j\}=0_N\oplus \{\alpha_j\}_{j>N}$ and $0_N$ denotes the zero sequence of length $N$. 

Lemma \ref{lemma3.1} below will be used repeatedly in Theorem \ref{THEOREM} which provides a necessary condition in terms of the weight sequence $\{\alpha_n\}$ for $\Sc(T, T^*)$ to be an SI semigroup when generated by a weighted shift.

\begin{lemma}\label{lemma3.1}
	Let $T$ be the  weighted shift with a complex weight sequence $\{\alpha_n\}$. For any $m,n\geq 1~\mbox{and}~i\geq 1$, if ${T^*}^mT^ne_i\neq 0$, then 
	\begin{enumerate}[label=(\roman*)]
		\item	$n-m\geq 1-i$, and
		\item ${T^*}^mT^ne_i=(c_i\alpha_i\bar{\alpha}_{i+n-m}) e_{i+n-m}$ where $c_i$ is the product of some  $\alpha_j$'s and $\bar{\alpha}_j$'s with indices $j > \min\{i,i+n-m\}$.
	\end{enumerate}
\end{lemma}	
\begin{proof}
	Since $T$ is the  weighted shift with a complex weight sequence $\{\alpha_k\}$, for readers' convenience we display again Equations (\ref{def1'})-(\ref{def2'}) for $k \geq 1$ and for $m \geq 1$: 
	\begin{equation} \label{def1}
	T^me_k=(\Pi_{j=k}^{k+m-1}\alpha_j)e_{k+m}
	\end{equation} and 
	\begin{equation}\label{def2}
	{T^*}^me_k=
	\begin{cases} 
	0 & \quad \text{ for } 1\leq k \leq m\\
	(\Pi_{k-m}^{k-1}\bar{\alpha}_j)e_{k-m} & \quad \text{ for } k \geq m+1 \\ 
	\end{cases}
	\end{equation}
	For $n\geq 1$, by (\ref{def1}), $T^ne_i=(\alpha_i\alpha_{i+1}\ldots\alpha_{i+n-1})e_{i+n}$, and so for $m \ge 1$, ${T^*}^mT^ne_i=(\alpha_i\alpha_{i+1}\cdots\alpha_{i+n-1}){T^*}^me_{i+n}$. 
	Since ${T^*}^mT^ne_i\neq 0,~{T^*}^me_{i+n}\neq 0$. 
	Therefore from (\ref{def2}), $i+n\geq m+1$, that is, $n-m\geq 1-i$ (which proves (i)) and $${T^*}^me_{i+n}=(\bar{\alpha}_{i+n-m}\bar{\alpha}_{i+n-m+1}\cdots\bar{\alpha}_{i+n-1})e_{i+n-m}.$$ Hence, 
	$${T^*}^mT^ne_i=(\alpha_i\alpha_{i+1}\cdots\alpha_{i+n-1})(\bar{\alpha}_{i+n-m}\bar{\alpha}_{i+n-m+1}\cdots\bar{\alpha}_{i+n-1})e_{i+n-m}.$$
	Note that the index $i$ is the smallest index of $\alpha_j$ in the first parenthesis of scalars and $i+n-m$ is the smallest index of $\bar{\alpha}_j$ in the second parenthesis of scalars in the above display. So, combining all the scalars $\alpha_j$'s and $\bar{\alpha}_j$'s together except $\alpha_i$ and $\bar{\alpha}_{i+n-m}$, we re-write 
	$${T^*}^mT^ne_i=(c_i\alpha_i\bar{\alpha}_{i+n-m})e_{i+n-m}$$ 
	where $c_i$ is the product of some $\alpha_j$'s and $\bar{\alpha}_j$'s with indices $j > \min\{i,i+n-m\}$. This completes the proof of the lemma.
\end{proof}
The next theorem is a necessary condition for $\Sc(T,T^*)$ to be SI: Each weight passed the first has its reciprocal consisting of products of later weights and their conjugates (not necessarily strictly later).
\begin{theorem}\label{THEOREM}
	Let $T$ be the weighted shift with complex weights $\alpha_j \neq 0$ for all $j \geq 1$. If $\Sc(T,T^*)$ is an SI semigroup, then for each $i \geq 2$, the reciprocal $1/\alpha_i$ is a product of some $\alpha_j$'s and $\bar{\alpha}_j$'s with indices $j\geq i$.
\end{theorem}
\begin{proof}
		Suppose $\Sc(T,T^*)$ is an SI semigroup. Then the principal ideal $(T)_{\Sc(T,T^*)}$ is selfadjoint. So, $T^*=XTY$ for some $X,Y\in \Sc(T,T^*)\cup\{I\}$ where either $X$ or $Y \neq I$, otherwise $T$ would be selfadjoint contradicting the nonselfadjointness of $T$. Moreover, $\range T^*\nsubseteq \range T$ (as $e_1 \in \range T^*$ but $e_1 \not \in \range T$). And hence $T^*=XTY$ implies that $X \ne I$ and must start with $T^*$, since otherwise it starts with $T$ implying $\range T^*\subseteq \range T$, a contradiction.
		
		Recall the obvious semigroup description for $\Sc(T,T^*)$ \cite[Proposition 1.6]{PW21}:
	
		\noindent $\Sc(T,T^*) =  \{T^n, {T^*}^n,  \Pi_{j=1}^{k}{T^*}^{m_j}T^{n_j},
	(\Pi_{j=1}^{k}{T^*}^{m_j}T^{n_j}){T^*}^{m_{k+1}},\Pi_{j=1}^{k}T^{n_j}{T^*}^{m_j}, (\Pi_{j=1}^{k}T^{n_j}{T^*}^{m_j})T^{n_{k+1}}\},$ where $n \ge 1,\,  k\ge1,\, n_j, m_j \ge 1\, \text{ for }\, 1 \le j \le k \text{ and } n_{k+1},m_{k+1}\geq 1$.
	
		Since $XTY$ is a word in $T$ and $T^*$ that starts with $T^*$ and has a $T$ in it, observing this semigroup list, $XTY$ can only have the third or 
	fourth form of the list. 
	Considering these two cases we obtain the necessary reciprocal condition.
	
	\textit{\textbf{Case 1:}} Suppose  
	$T^* = XTY=\Pi_{j=1}^{k}({T^*}^{m_j}T^{n_j})$ 
	for some $k\geq 1$ and $m_j, n_j \geq 1$ for $1\leq j \leq k$. 
	Therefore, for all $i > 1$, one has $0 \ne T^*e_i=\Pi_{j=1}^{k}({T^*}^{m_j}T^{n_j})e_i$. In particular, for $i\geq 2$, by Equation (\ref{W0}), 
	\begin{equation}\label{eee1}
	0 \ne \bar{\alpha}_{i-1}e_{i-1}=T^*e_i=\Pi_{j=1}^{k}({T^*}^{m_j}T^{n_j})e_i.
	\end{equation}
	
	Since $\bar{\alpha}_{i-1}e_{i-1} \neq 0$ for all $i \geq 2$, $\Pi_{j=1}^{k}({T^*}^{m_j}T^{n_j})e_i\neq 0$ for all $i \geq 2$.
	Hence, ${T^*}^{m_k}T^{n_k}e_i\neq 0$ for all $i\geq 2$. So, by Lemma~\ref{lemma3.1}(i), $n_k-m_k\geq 1-i$  for all $i\geq 2$. In particular, for $i=2$, one obtains $n_k-m_k\geq -1$ and for $i \geq 2$, by Lemma \ref{lemma3.1}(ii),
	$$0 \ne {T^*}^{m_k}T^{n_k}e_i= (c_{i,k}\alpha_i\bar{\alpha}_{i+n_k-m_k})e_{i+n_k-m_k},$$ where $c_{i,k}$ 
	depends on $e_i$ and $k$ as in the $k$-product form for $T^*$, and $c_{i,k}$ is a product of some $\alpha_j$'s and $\bar{\alpha}_j$'s with indices 
	$j > \min\{i,i+n_k-m_k\}$. Since $n_k-m_k\geq -1$, one has these indices $j > i-1$. 
	
	Thus starts a backwards induction. That is, if $k > 1$, we next consider $({T^*}^{m_{k-1}}T^{n_{k-1}})({T^*}^{m_k}T^{n_k})e_i$. 
	Since $\Pi_{j=1}^{k}({T^*}^{m_j}T^{n_j})e_i\neq 0$ for all $i \geq 2$, one has 
	$$0 \neq ({T^*}^{m_{k-1}}T^{n_{k-1}})({T^*}^{m_k}T^{n_k})e_i
	=(c_{i,k}\alpha_i\bar{\alpha}_{i+n_k-m_k}){T^*}^{m_{k-1}}T^{n_{k-1}}e_{i+n_k-m_k}.$$
	Hence ${T^*}^{m_{k-1}}T^{n_{k-1}}e_{i+n_k-m_k} \ne 0.$
	Again by Lemma~\ref{lemma3.1}(i), $n_{k-1}-m_{k-1}\geq 1-(i+n_k-m_k)$, i.e, $(n_k-m_k)+(n_{k-1}-m_{k-1})\geq 1-i$ for all $i\geq 2$. In particular, for $i=2$, one obtains $(n_k-m_k)+(n_{k-1}-m_{k-1})\geq -1$. And again by Lemma~\ref{lemma3.1}(ii), for $i \geq 2$, 
	$${T^*}^{m_{k-1}}T^{n_{k-1}}e_{i+n_k-m_k}=(c_{i,k-1}\alpha_{i+n_k-m_k}\bar{\alpha}_{i+\sum_{j=k-1}^{k}(n_j-m_j)})e_{i+\sum_{j=k-1}^{k}(n_j-m_j)},$$ 
	where $c_{i,k-1}$ is a product of some $\alpha_j$'s and $\bar{\alpha}_j$'s  with indices $j > \min\{i+(n_k-m_k),i+(n_k-m_k)+(n_{k-1}-m_{k-1})\}$. Since $n_k-m_k \geq -1$ and $(n_k-m_k)+(n_{k-1}-m_{k-1})\geq -1$, so we have the indices $j > i-1$. Therefore,
	\begin{align*} ({T^*}^{m_{k-1}}T^{n_{k-1}})&({T^*}^{m_k}T^{n_k})e_i=(c_{i,k}\alpha_i\bar{\alpha}_{i+(n_k-m_k)})
	(c_{i,k-1}\alpha_{i+(n_k-m_k)}\bar{\alpha}_{i+\sum_{j=k-1}^{k}(n_j-m_j)})e_{i+\sum_{j=k-1}^{k}(n_j-m_j)}\\
	&=(c'_{i,k-1}\alpha_i|\alpha_{i+n_k-m_k}|^2\bar{\alpha}_{i+\sum_{j=k-1}^{k}(n_j-m_j)})e_{i+\sum_{j=k-1}^{k}(n_j-m_j)},
	\end{align*}
	where $c'_{i,k-1} :=c_{i,k-1}c_{i,k}$ is a product of some $\alpha_j$'s and $\bar{\alpha}_j$'s with indices $j > i-1$. Continuing backwards in this way and applying Lemma \ref{lemma3.1}(i)-(ii) repeatedly, one finally obtains, for $i\geq 2$,
	
	\begin{equation}\label{eee2}
	\Pi_{j=1}^{k}({T^*}^{m_j}T^{n_j})e_i=(\delta_i\alpha_i|\alpha_{i+n_k-m_k}|^2\cdots|\alpha_{i+\sum_{j=2}^{k}(n_j-m_j)}|^2\bar{\alpha}_{i+\sum_{j=1}^{k}(n_j-m_j)})e_{i+\sum_{j=1}^{k}(n_j-m_j)}
	\end{equation}
	where $\delta_i$ is a product of some $\alpha_j$'s and $\bar{\alpha}_j$'s  with indices $j > i-1$. Moreover, for each $1 \leq r \leq k$, one has
	$$\sum_{j=r}^{k}(n_j-m_j)\geq-1.$$
	
	Recall from Equation (\ref{eee1}): $$\bar{\alpha}_{i-1}e_{i-1}=\Pi_{j=1}^{k}({T^*}^{m_j}T^{n_j})e_i$$ for all $i\geq 2$. By replacing the right-hand side with the expression obtained in Equation (\ref{eee2}), for all $i\geq 2$ one has $$\bar{\alpha}_{i-1}e_{i-1}=(\delta_i\alpha_i|\alpha_{i+n_k-m_k}|^2\cdots|\alpha_{i+\sum_{j=2}^{k}(n_j-m_j)}|^2\bar{\alpha}_{i+\sum_{j=1}^{k}(n_j-m_j)})e_{i+\sum_{j=1}^{k}(n_j-m_j)}.$$ 
	Equating subscripts and scalars we obtain $i-1 = i + \sum_{j=1}^{k}(n_j-m_j)$ and hence $\sum_{j=1}^{k}(n_j-m_j)=-1$, so at least one $\bar{\alpha}_{i-1}$ appears in this product;  then letting $s$ denote the number of $\bar{\alpha}_{i-1}$ appearing in this product (so $s \ge 1$ and depends only on the $m_j. n_j$'s, hence is independent of $i$); 
	so for all $i \geq 2$ one obtains, 
	$$\bar{\alpha}_{i-1}=\bar{\alpha}_{i-1}^{s}\alpha_i\gamma_i  \quad \text{and} \quad 
	\bar{\alpha}_{i}=\bar{\alpha}_{i}^{s}\alpha_{i+1}\gamma_{i+1},$$  
	where $s \geq 1$ and $\gamma_i$ is the product of $\alpha_j$'s and $\bar{\alpha}_j$'s with $j > i -1$. Since $\bar{\alpha}_{i-1} \neq 0$, by considering the cases $s=1$ and $s>1$ separately, a necessary condition is that $1/\alpha_i$ is a product of $\alpha_j$'s and $\bar{\alpha}_j$'s  with indices $j > i-1$.
	Indeed, if $s = 1$ then the first identity yields that $1/\alpha_i$ as a product of some $\alpha_j,\bar{\alpha}_j, j \ge i$; and if $s \ge 2$, then taking conjugates on both sides in the second identity yields this fact. 
	
	
	\textit{\textbf{Case 2:}} Suppose $T^* = XTY=\Pi_{j=1}^{k}({T^*}^{m_j}T^{n_j}){T^*}^{m_{k+1}}$ for some $k\geq 1$, $m_j, n_j \geq 1$ for $1\leq j \leq k$, and  $m_{k+1} \geq 1$. We first claim that $m_{k+1} =1$.
	Since $T^*e_2\neq 0$, one has $(\Pi_{j=1}^{k}{T^*}^{m_j}T^{n_j}){T^*}^{m_{k+1}}e_2 \neq 0$ and so ${T^*}^{m_{k+1}}e_2 \neq 0$. Then it follows from Equation (\ref{def2}) that $m_{k+1}=1$. Therefore, $T^*=\Pi_{j=1}^{k}({T^*}^{m_j}T^{n_j})T^*$. 
	In particular by Equation (\ref{W0}), for $i\geq 2$,
	\begin{equation}\label{aaa1} 
	\bar{\alpha}_{i-1}e_{i-1} = T^*e_i= 	\bar{\alpha}_{i-1}\Pi_{j=1}^{k}({T^*}^{m_j}T^{n_j})e_{i-1}.
	\end{equation}
	For $i \geq 2$, since $\bar{\alpha}_{i-1} \neq 0$, ${T^*}^{m_k}T^{n_k}e_{i-1} \neq 0$ and so, by Lemma \ref{lemma3.1}(i) applied to $i-1$, one obtains $n_k-m_k \geq 1-(i-1)$. Hence for $i=2$ one has $n_k-m_k\geq 0$, and for $i\geq 2$ one has by Lemma \ref{lemma3.1}(ii)
	$$ {T^*}^{m_k}T^{n_k}e_{i-1}= (c_{i-1,k}\alpha_{i-1}\bar{\alpha}_{i-1+n_k-m_k})e_{i-1+n_k-m_k},$$ where $c_{i-1,k}$ is a product of some $\alpha_j$'s and $\bar{\alpha}_j$'s with indices $j > \min\{i-1,i-1+n_k-m_k\} = i - 1$, the latter equality since $n_k-m_k\geq 0$, and so indices $j > i-1$. Following the same backwards induction argument as in Case $1$ for $\Pi_{j=1}^{k}({T^*}^{m_j}T^{n_j})e_{i-1} \neq 0$, Equation (\ref{eee2}) in this case becomes 
	\begin{equation}\label{eee3}
	\Pi_{j=1}^{k}({T^*}^{m_j}T^{n_j})e_{i-1}=(\delta_{i-1}\alpha_{i-1}|\alpha_{i-1+n_k-m_k}|^2\cdots|\alpha_{i-1+\sum_{j=2}^{k}(n_j-m_j)}|^2\bar{\alpha}_{i-1+\sum_{j=1}^{k}(n_j-m_j)})e_{i-1+\sum_{j=1}^{k}(n_j-m_j)}
	\end{equation}	
	where $\delta_{i-1}$ is a product of some $\alpha_j$'s and $\bar{\alpha}_j$'s with indices $j > i-1$. Moreover, for each $1 \leq r \leq k$, one has
	$$\sum_{j=r}^{k}(n_j-m_j)\geq 0.$$
	Substituting in Equation (\ref{aaa1}) the expression obtained in Equation (\ref{eee3}), we obtain
	\begin{equation}\label{aaa2} 
	\bar{\alpha}_{i-1}e_{i-1} = \bar{\alpha}_{i-1}(\delta_{i-1}\alpha_{i-1}|\alpha_{i-1+n_k-m_k}|^2\cdots|\alpha_{i-1+\sum_{j=2}^{k}(n_j-m_j)}|^2\bar{\alpha}_{i-1+\sum_{j=1}^{k}(n_j-m_j)})e_{i-1+\sum_{j=1}^{k}(n_j-m_j)}.
	\end{equation}
	Equating subscripts and scalars, 
	we obtain $i-1 = i -1+ \sum_{j=1}^{k}(n_j-m_j)$, that is, $\sum_{j=1}^{k}(n_j-m_j)=0$, 
	so along with at least one $\alpha_{i-1}$ we have at least two $\bar{\alpha}_{i-1}$ appears in this product;  then letting $s$ denote the number of 
	$\bar{\alpha}_{i-1}$ appearing in this product (so $s \ge 2$ and depends only on the $m_j. n_j$'s, hence is independent of $i$); 
	so for all $i \geq 2$ one obtains, 
	$$\bar{\alpha}_{i-1}=\alpha_{i-1}\bar{\alpha}_{i-1}^{s}\gamma_{i},$$
	where $\gamma_{i}$ is the product of $\alpha_j$'s and $\bar{\alpha}_j$'s with $j > i-1$ and $s \geq 2$. 
	Then since $\bar{\alpha}_{i-1} \neq 0$, $1/\alpha_{i-1}=\bar{\alpha}_{i-1} ^{s-1}\gamma_{i}$ for all $i \ge 2$ or equivalently, 
	\begin{equation*}
	1/\alpha_i={\bar{\alpha}_i}^{s-1}\gamma_{i+1} \quad \text{for all~} i \ge 1,
	\end{equation*}
	where $\gamma_{i+1}$ is the product of $\alpha_j$'s and $\bar{\alpha}_j$'s with $j > i$. 
	
	Therefore in each of the only two cases possible, Case 1 and Case 2, we obtained the necessary condition stated in the theorem. This completes the proof.
\end{proof}

\begin{remark}\label{alpha1}
	We emphasize here that one cannot infer the reciprocal of the first weight $\alpha_1$ is a product of its later $\alpha_j$'s,  if $\Sc(T, T^*)$ is SI. For example, consider $T$ with the weight sequence $\left(2, 1, 1, \cdots \right)$. Then by direct computation, $T$ satisfies the equation $(T^*T)T = T$ and hence, by Proposition \ref{quasi-isometries} (see below), $\Sc(T, T^*)$ is simple, a special case of SI. But $\alpha_1 = 2$ clearly cannot  not have its inverse as a product of $\alpha_j$'s with index $j \geq 1$.
	
	Moreover, easy examples of non-SI semigroups $\Sc(T,T^*)$, i.e., where the necessary reciprocal condition in Theorem \ref{THEOREM} fails, abound, as for instance all multiples of the unilateral shift $cS$, with $|c| \ne 1$. So also for all weighted shifts with absolute values of all the weights less than one (or bigger than one).
	
	In \cite{PW21}, we obtained the necessary norm condition $||T||\ge 1$ (as a consequence of [\cite{PW21}, Remark 1.22 (iii), see also Example 1.23]) for $\Sc(T,T^*)$ to be an SI semigroup for the more general class of nonselfadjoint operators $T$. And here in Theorem \ref{THEOREM}, we obtained the necessary reciprocal condition for the class of weighted shift operators $T$ with all nonzero weights, which is a subclass of nonselfadjoint operators. We next show that for this class of weighted shift operators, the reciprocal condition is stronger than the norm condition. For that we need to show the reciprocal condition fails whenever the norm condition fails. Suppose $T$ is a weighted shift with $||T||< 1$, then $0<|\alpha_n|<1$ for all $n\ge 1$. Then clearly $1/{\alpha_2}$ (with $1/{|\alpha_2|}>1$) cannot be a product of some powers of $\alpha_j$'s and $\bar{\alpha}_j$'s with $j\ge 2$ because such a product would have absolute value smaller than $1$. Therefore the reciprocal condition fails.
	
	Also one can construct easy examples of weighted shifts where the necessary reciprocal condition fails but the necessary norm condition is satisfied. For instance, consider weighted shifts $T$ with weight sequences $\{1-1/n\}$ and $\{1+1/n\}$. For both these weighted shifts, $T$ satisfies  the necessary norm condition $||T||\ge 1$, however the reciprocal condition clearly fails so  $\Sc(T,T^*)$ are not SI by  Theorem \ref{THEOREM}. This also shows that, for the class of weighted shifts, the condition $||T||\ge 1$ is necessary but not sufficient for $\Sc(T,T^*)$ to be an SI semigroup. 
\end{remark}

The necessary reciprocal condition obtained in Theorem \ref{THEOREM} is also \textit{not} sufficient for $\Sc(T, T^*)$ to be SI as shown in the following example.

	\begin{example}\label{p-example}
		The weighted shift $T$ with weight sequence $\{2, 1/\sqrt 2, 2,1/\sqrt 2, \ldots\}$ satisfies the necessary reciprocal condition, but the semigroup generated by $T$ is not SI. Indeed, this weight sequence is periodic with period $2$ but $T^2$ is not an isometry and hence by Theorem~\ref{periodic SI characterization},  $\Sc(T, T^*)$ is not SI.
	\end{example}
	
Although we could not find a sufficient condition for $\Sc(T, T^*)$ to be SI, nevertheless we were able to obtain a necessary and sufficient condition for $\Sc(T, T^*)$ to be SI when generated by two particular subclasses of weighted shift operators from among those that have no zero gap, that is, $\{\alpha_j\}=0_N\oplus \{\alpha_j\}_{j>N}$ where $\alpha_j\ne 0$ for $j>N\ge 0$ (recalling that for the zero gap case $\Sc(T, T^*)$ is already SI characterized in Theorem~\ref{theorem 3.1}). Those classes are: those weighted shifts whose nonzero weights $\{\alpha_j\}_{j>N}$ have periodic absolute value sequence ($\{|\alpha_j|\}_{j>N}$); and those weighted shifts whose nonzero weights $\{\alpha_j\}_{j>N}$ have eventually constant absolute value sequence ($\{|\alpha_j|\}_{j>N}$) (Theorem~\ref{periodic SI characterization}, Corollary \ref{eventually-constant}). Observe that the first class properly contains all weighted shifts with periodic weight sequence and the second class properly contains all weighted shifts with eventually constant weight sequence. For the larger class of weighted shifts having weight sequences with absolute values almost periodic (see Definition \ref{definition almost periodic}), we obtained a necessary condition which is not sufficient for the $\Sc(T, T^*)$ to be SI (Theorem \ref{almost periodic theorem}).\\	
	
	\noindent \textbf{A characterization of SI semigroups $\Sc(T,T^*)$ generated by weighted shifts with periodic attributes.}
	
	We have seen in Theorem \ref{theorem 3.1} a characterization of SI semigroups $\Sc(T,T^*)$ for those weighted shifts whose weight sequence has the gap property that $\alpha_i \neq 0, \alpha_{i+1}=0$ for some $i\ge 1$.
	
	We could not achieve both necessary and sufficient conditions for the class of weight sequences without a zero gap (i.e, of the form $0_k \oplus \{\alpha_i\}$ where $\alpha_i \neq 0$ for all $i$ and $k\geq 0$), but among this class we will next determine necessary and sufficient conditions to ensure the SI property for semigroups $\Sc(T,T^*)$  generated by weighted shifts where the sequence $\{|\alpha_i|\}$ is periodic. And we will see (in the begining of the proof of Theorem~\ref{first p-theorem}) that this characterization reduces to the case where $k=0$ and the weight sequence $\{\alpha_i\}$ has strictly positive periodic weights. To obtain this SI characterization we need some facts about diagonal matrices (that is, strictly upper, strictly lower and main diagonals), which are discussed in the next proposition and two corollaries.
	\color{black}

	\vspace{.2cm}
	
	\noindent  \textit{Preliminaries on diagonal matrices - upper and lower diagonals}

	Call the set $\mathfrak D$ of all matrices with at most one nonzero diagonal. That is,  lower diagonals (include the possibility of main diagonals), upper diagonals (include the possibility of main diagonals), strictly lower and strictly upper diagonal matrices. More precisely, for $\{e_n\}$ an orthonormal basis of $\mathcal H$, by a $k$-diagonal matrix with complex weights $\alpha=\{\alpha_j\}$, denoted by $D^{(\alpha)}_{k}$, we mean 
	
	\begin{equation} \label{diagonal definition}
	D^{(\alpha)}_{k}e_i =
	\begin{cases}
	&\alpha_i e_{i+k}, ~\text{for}~k\ge 0 \quad \\
	&\text{and for}~k<0, \\
	&0, ~\text{for}~1 \le i \le -k \\
	&\alpha_{i+k} e_{i + k},~\text{for}~ i > -k
	\end{cases}
	\end{equation}
	That is, $k >0, k = 0, k < 0$ corresponds respectively to strictly lower, main and strictly upper diagonal matrices. 
	\begin{definition}\label{definition of eventually periodic}
		For a sequence $\alpha=\{\alpha_j\}$, we say $\alpha$ is eventually periodic with period $p$, if there exists an $N \geq 0$ such that $\alpha_{j+N} = \alpha_{j+N+p}$ for $j \geq 1$. In particular, a periodic sequence with period $p$ is the special sequence with $N = 0$.
	\end{definition}

	In the next proposition and proof we denote maximum and minimum of integers by $\vee, \wedge$ respectively. And in Equation (\ref{LU}) the direct sum indicates we start the sequence with that number of zeros.
	
	\begin{proposition} \label{multiplication}
		Let $\mathfrak D = \{D^{(\alpha)}_{k}: k \in \mathbb Z, \alpha \in \ell^{\infty}\}$. Then $D^{(\alpha)}_{k}D^{(\beta)}_{l} = D^{(\gamma)}_{k+l}$ for $\gamma \in \ell^{\infty}$ given by: 
		\begin{equation} \label{LL,UU,UL}
		\gamma =  \{\beta_i \alpha_{i+l}\}, \, \{\beta_{i+l} \alpha_{i+l+k}\},\, \{\beta_i \alpha_{i+l+k}\}~ \text{respectively for $~k,l \ge 0,~ k,l < 0$, and $k<0 , ~l \ge 0$}\end{equation}
		and  
		\begin{equation} \label{LU}
		\gamma = 0_{k \wedge -l} \oplus \{\beta_{i+l} \alpha_{i+l}\}_{i > -l}\quad \text{for $~k \ge 0, ~l < 0$}
		\end{equation}
		Consequently, $\mathfrak D$ forms a multiplicative semigroup.
	\end{proposition}
	\begin{proof}
		Using Equation (\ref{diagonal definition}) to compute $D^{(\alpha)}_{k}D^{(\beta)}_{l}$ one obtains:
		\begin{equation} \label{k,l ge 0}
		\text{For}~ k,l \ge 0, i \ge 1, \quad D^{(\alpha)}_{k}D^{(\beta)}_{l} e_i = D^{(\alpha)}_{k}(\beta_i e_{i+l})  = \beta_i \alpha_{i+l} e_{i + k + l}
		\end{equation}
		\begin{equation}  \label{k,l < 0}
		\text{For}~ k,l < 0, i > -k - l, \quad D^{(\alpha)}_{k}D^{(\beta)}_{l} e_i =  \beta_{i+l} \alpha_{i+l+k} e_{i + k + l}, \quad \text{and $0$ for $1 \le i \le -k-l$}
		\end{equation}
		For $k < 0, l \ge 0,  i > (-k - l) \vee 0$, 
		\begin{equation}  \label{k < 0, l ge 0}
		D^{(\alpha)}_{k}D^{(\beta)}_{l} e_i =  D^{(\alpha)}_{k}(\beta_i e_{i+l}) = \beta_i \alpha_{i+l+k} e_{i+l+k},\quad \text{and $0$ for $1 \le i \le -k-l$}
		\end{equation}
		For $k \ge 0, l < 0, i > -l$
		\begin{equation} \label{k ge 0, l < 0}
		D^{(\alpha)}_{k}D^{(\beta)}_{l} e_i = D^{(\alpha)}_{k}(\beta_{i+l} e_{i+l}) 
		= \beta_{i+l} \alpha_{i+l} e_{i + k + l}, \quad \text{and $0$ \text{for} $1 \le i \le -l$} 
		\end{equation}\end{proof}
	
	Observe that the sets of upper (lower) diagonals and the sets of strictly upper (strictly lower) diagonals form subsemigroups of $\mathfrak D$. Moreover, their subsets with eventually periodic weight sequences also form subsemigroups of $\mathfrak D$.  Indeed, all this follows naturally from Equations (\ref{k,l ge 0})-(\ref{k ge 0, l < 0}) focusing on the weight sequence of the products Equations (\ref{LL,UU,UL})-(\ref{LU}). 
	
	However, regarding the two classes in $\mathfrak D$ with periodic and eventually periodic weight sequences, the first is not a subsemigroup because of Equation (\ref{LU}). But the eventually periodic ones are.
	To see that those elements in $\mathfrak D$ with eventually periodic weights is a multiplicative semigroup, observe that the product of two periodic sequences has period at least the smallest common multiple of their periods (possibly smaller).

	Then in particular, for a weighted shift operator $T$ with periodic weight sequence, so a strictly lower diagonal matrix, every word in $\Sc(T, T^*)$ is eventually periodic with the same period. To codify,
	
	\begin{corollary} \label{semigroups of periodic diagonals}
		If $T$ is a weighted shift in  $\mathfrak D$ with periodic (or eventually periodic) weight sequence of period $p$, then any $A\in \Sc(T,T^*)$ (any word, that is, any finite product of $T~\text{and/or}~T^*$) is a strictly lower, strictly upper or main diagonal with eventually periodic weight sequence with the same period $p$.
	\end{corollary}
	
	So in short, for $T$ a nonzero weighted shift, it is now clear that any $A\in \Sc(T,T^*)$ (i.e., word in $T$ and $T^*$) has matrix representation (with respect to a fixed orthonormal basis) with exactly one nonzero diagonal. 
	Furthermore, if $T$ has a periodic weight sequence, then the nonzero diagonal of $A$ has an eventually periodic weight sequence with the same period (possibly having some initial weights zero).
	
	To proceed with our SI characterization, we also need the concept of periodic mean. 
	\begin{definition}\label{first p-mean definition}
		For a sequence $\alpha=\{\alpha_j\}$ with period $p$, we define the periodic mean $q:= |\alpha_1\alpha_2\cdots\alpha_p|^{1/p}$. 
		For an eventually periodic sequence $\alpha=\{\alpha_j\}$ with period $p$ (defined in Definition \ref{definition of eventually periodic}), we define its periodic mean $q=|\alpha_{N+1}\alpha_{N+2}\cdots\alpha_{N+p}|^{1/p}.$ So the case $N=0$ is the periodic case.
	\end{definition} 
	For a periodic sequence $\{\alpha_n\}$ in $\mathbb{C}$  with a period $p$, we define \textit{the} periodic mean $q$ as $q:=|\alpha_1\alpha_2\cdots\alpha_p|^{1/p}$. At first glance this definition may seem not well-defined because if a sequence has a period, then it has many periods, for instance clearly all multiples of that period are also periods. So for well-definedness it suffices then to show $q$ is independent of all periods $p$. Indeed let $r$ be the smallest period of a periodic sequence. Then any period $p$ must be a multiple of $r$ because otherwise $p = mr + s$ for some $0< s < r$, 
	hence for all $i \ge 1$, $\alpha_i = \alpha_{i+p} = \alpha_{i+mr+s} = \alpha_{i+s}$ so $s$ is also a period, against the minimality of $r$.
	It follows that the periodic mean is independent of the choice of the period $p$ for the sequence because for any period $p$ of the sequence, $p=mr$ for some $m\ge 1$. 
	Therefore, $q=|\alpha_1\alpha_2\cdots\alpha_p|^{1/p}=|\alpha_1\alpha_2\cdots\alpha_{mr}|^{1/mr}=|(\alpha_1\alpha_2\cdots\alpha_r)^m|^{1/mr}=|\alpha_1\alpha_2\cdots\alpha_r|^{1/r}$.\color{black}
	
	Also note that for any two periodic sequences $\{\alpha_n\}, \{\beta_n\}$ with the same period $p$ and periodic means $q_1,q_2$ respectively, the product sequence $\{\alpha_n\beta_n\}$ is periodic with the same period $p$ and  periodic mean $q_1q_2$. Clearly, $\{\alpha_n\beta_n\}$ is periodic with  period $p$ and its periodic mean is given by:
	$$|(\alpha_1\beta_1)(\alpha_2\beta_2)\cdots(\alpha_p\beta_p)|^{1/p}=(|\alpha_1\alpha_2\cdots \alpha_p|)^{1/p}(|\beta_1\beta_2\cdots \beta_p|)^{1/p}=q_1q_2.$$
	Likewise for eventually periodic sequences.
	
	And once $\{\alpha_n\}$ is a periodic (or eventually periodic) sequence with period $p$ and periodic mean $q$, so also is its tail sequences ($\{\alpha_{n+l}\}_{l\ge 1}$) $p$-periodic (or eventually $p$-periodic) with periodic mean $q$. Then from Equations~(\ref{LL,UU,UL}) and (\ref{LU}) we obtain:
	
	\begin{corollary}\label{periodic mean of products}
		
		For any $D^{(\alpha)}_{k},D^{(\beta)}_{l}\in \mathfrak{D}$ with periodic (or eventually periodic) sequences $\alpha, \beta$ with the same period $p$ and periodic means $q_1, q_2$, respectively, the product diagonal $D^{(\alpha)}_{k}D^{(\beta)}_{l} = D^{(\gamma)}_{k+l}$ has $\gamma$ an eventually periodic sequence with period $p$ and periodic mean $q_1q_2$. 
	\end{corollary}

	In particular, given a weighted shift $T$ of periodic weight sequence with period $p$ and periodic mean $q$, then for any $A\in \Sc(T,T^*)$, in addition to being an eventually periodic diagonal with period $p$ (as discussed in Corollary~\ref{semigroups of periodic diagonals}), its periodic mean is $q^s$, where $s$ is the degree of $A$ (that is, the sum of the powers of $T$ and $T^*$ in the word $A$). To codify,
	
	\begin{proposition}\label{proposition3.9}
		Let $T$ be a weighted shift with periodic (or eventually periodic) weight sequence $\{\alpha_n\}$ with period $p$ and periodic mean $q$
		Then for $A\in \Sc(T,T^*)$ with $s = \text{degree}\,A$, the diagonal of $A$ has eventually periodic weight sequence of period $p$ with periodic mean $q^s$.
	\end{proposition}
	\begin{proof}
		It follows from Corollary \ref{semigroups of periodic diagonals} that $A$ is a diagonal (that is, a strictly upper or strictly lower or main) with eventually periodic weight sequence with period $p$. Also since both $T$ and $T^*$ are diagonals (strictly lower and strictly upper, respectively) with periodic mean $q$ and $A$ being a word in $T$ and $T^*$, it follows by applying induction on $s$ (the degree of $A$) and using Corollary \ref{periodic mean of products} that $A$ is eventually periodic with period $p$ and has periodic mean $q^s$.
	\end{proof}
	In what follows we denote the zero matrix in $M_k(\mathbb{C})$ by $0_k$, and we use the same symbol to denote the zero sequence of length $k$ as well depending on the obvious context. 
	
	Before we can give the SI characterzation for $\Sc(T,T^*)$ generated by a weighted shift with weight sequence $0_k\oplus \{\alpha_n\}$ such that $\{|\alpha_n|\}$ is a $p$-periodic sequence of nonzero numbers as promised in this periodic subsection, we need the following result concerning periodic means. Also in its proof we will see how the SI characterization for the more general class (that is, with weight sequence $0_k\oplus \{\alpha_n\}$ where $\{|\alpha_n|\}$ is $p$-periodic) reduces to the SI chacterization for periodic weight sequences of strictly positive weights.
	\begin{theorem}\label{first p-theorem}
		Let $T$ be a weighted shift with weights $0_k\oplus \{\alpha_n\}$ where $\{|\alpha_n|\}$ is a $p$-periodic sequence of nonzero numbers with periodic mean $q$. If $\Sc(T,T^*)$ is an SI semigroup, then $|\alpha_1\alpha_2\cdots\alpha_p|=1$ (i.e., $q=1$), or equivalently, $T^p=0_k\oplus U$ with $U$ an isometry.
	\end{theorem}
	\begin{proof}
		First we reduce $T$ to the case where all weights are strictly positive. Given that $T$ has initial $k$-weights zero, so $T=0_k\oplus T_1$ where $T_1$ is a weighted shift with all nonzero weights $\{\alpha_n\}$. Therefore clearly $\Sc(T,T^*)$ is an SI semigroup if and only if $\Sc(T_1,T_1^*)$ is SI. And $T^p=0_k\oplus U$ with $U$ an isometry if and only if $T^p_1$ is an isometry. Also it is straight forward to check that $T^p_1$ is an isometry if and only if  $|\alpha_1\alpha_2\cdots\alpha_p|=1.$ Furthermore $T_1$ is unitarily equivalent to a weighted shift $S$ with $p$-periodic weight sequence $\{|\alpha_n|\}$  (see \cite[Problem 89]{Hal82}). Therefore $\Sc(T_1,T^*_1)$ is SI if and only if  $\Sc(S,S^*)$ is SI. And the property of being an isometry is preserved under unitary equivalence. 
		Hence in order to prove that $\Sc(T,T^*)$ being SI implies  $T^p_1$ is an isometry, it suffices to prove the same for $S$. That is, without loss of generality we can assume that $T$ has $p$-periodic strictly positive weight sequence $\{\alpha_n\}$.
		
		Suppose $\Sc(T,T^*)$ is an SI semigroup. Then the principal ideal $(T)_{\Sc(T,T^*)}$ is selfadjoint. Therefore, $T^*=XTY$ for some $X,Y\in \Sc(T,T^*)\cup \{I\}$, where $X,Y$ cannot both be the identity operator $I$ because $T$ is nonselfadjoint. Since $T^*$ is a strictly upper diagonal matrix with $p$-periodic weight sequence and $T^*=XTY$, so is $XTY$. Also, $XTY$ being a finite product of $T,T^*$, the sum $s$ of the powers of $T$ and $T^*$ in $XTY$ is greater than one as $X$ or $Y$ is not the identity operator. Then it follows from Proposition \ref{proposition3.9} that the periodic weight sequence of $XTY$ in $T^*=XTY$ must have the periodic mean $q^s$, where $s\ge 2$.
		Then because all $\alpha_n$ are nonzero, $q \ne 0$, and because $T^*=XTY$, their weighted sequences have the same periodic mean, so $q = q^s$ implying $q=1$ which further implies that  $\alpha_1\alpha_2\cdots\alpha_p=1$. 
		
		To see that the condition $\alpha_1\alpha_2\cdots\alpha_p=1$ is equivalent to $T^p$ being an isometry, observe that in Equations (\ref{def1'})-(\ref{def2'}) for $m=p$, both products remain constant when the sequence is $p$-periodic and observe that ${T^p}^*T^p =(\alpha_1\alpha_2\cdots\alpha_p)^2\cdot I$, which shows  $T^p$ is an isometry if and only if $\alpha_1\alpha_2\cdots\alpha_p=1$.
\end{proof}
	Next we prove the converse of above Theorem~\ref{first p-theorem}. That is, for a given weighted shift $T$ for which $T^p=0_k\oplus U$ with $U$ an isometry, $\Sc(T,T^*)$ is an SI semigroup and in fact it is simple (see Theorem~\ref{periodic SI characterization}). Towards proving this, observe that since $T$ has initial $k$-weights zero, we can write $T=0_k\oplus T_1$, where $T_1$ is a weighted shift with all nonzero weights. Then clearly the SI property and simplicity of $\Sc(T,T^*)$ is equivalent to the SI property and simplicity of $\Sc(T_1,T_1^*)$ respectively. Also the condition that $T^p=0_k\oplus U$ with $U$ an isometry is equivalent to $T^p_1$ being an isometry. Therefore in proving Theorem \ref{periodic SI characterization},without loss of generality we can assume that $T^p$ is an isometry. In the next proposition, we prove the simplicity (hence the SI property) of the semigroup  $\Sc(T,T^*)$ generated by a weighted shift $T$ for which $T^p$ an isometry.	
\begin{proposition}\label{p-isometry}
Let $T$ be a weighted shift for which $T^p$ is an isometry for some $p \geq 1$. Then $\Sc(T,T^*)$ is simple.
\end{proposition}
\begin{proof}
In order to prove the simplicity of $\Sc(T,T^*)$, it suffices to prove that all its principal ideals coincide with $\Sc(T,T^*)$. We do this by showing the principal ideal generated by each of the six forms of the semigroup list coincide with the semigroup $\Sc(T, T^*)$. Recall the semigroup list for $\Sc(T,T^*)$ \cite[Proposition 1.6]{PW21}:
\noindent $\Sc(T,T^*) =  \{T^n, {T^*}^n,  \Pi_{j=1}^{k}{T^*}^{m_j}T^{n_j},
		(\Pi_{j=1}^{k}{T^*}^{m_j}T^{n_j}){T^*}^{m_{k+1}},\Pi_{j=1}^{k}T^{n_j}{T^*}^{m_j},$  
		$(\Pi_{j=1}^{k}T^{n_j}{T^*}^{m_j})T^{n_{k+1}}\},$ where $n \ge 1,\,  k\ge1,\, n_j, m_j \ge 1\, \text{ for }\, 1 \le j \le k \text{ and } n_{k+1},m_{k+1}\geq 1$.
		
First we prove the principal ideals generated by the first and second forms coincide with $\Sc(T,T^*)$. Since ${T^*}^pT^{p}=I$, using induction one also has ${T^*}^{mp}T^{mp} = I$ for each $m\geq 1$. So $I \in (T^m)_{\Sc(T, T^*)}$ and $I \in ({T^*}^m)_{\Sc(T, T^*)}$ for each $m\geq 1$ and hence, 
\begin{equation}\label{ideals equation}
		(T^m)_{\Sc(T, T^*)} = ({T^*}^m)_{\Sc(T, T^*)} = \Sc(T, T^*)\qquad \text{for all}~m\geq 1.
\end{equation} 
		Next we prove that the principal ideal $\mathcal{J}$ generated by any of the last four forms in the above semigroup list contains an operator ${T^*}^m$ for some $m \ge 1$. Then $\mathcal{J}$ would contain the principal ideal $({T^*}^m)_{\Sc(T, T^*)}$, which is $\Sc(T, T^*)$ by Equation (\ref{ideals equation}). But first observe that every principal ideal generated by a fourth or fifth or sixth form contains an operator of the third form. 
Because multiplying on the left or right or both sides of the operator that generates the principal ideal, by $T$ or $T^*$, one can obtain the operator of the third form. Therefore it suffices to prove that the principal ideals generated by each of the third form contain ${T^*}^m$ for some $m \ge 1$.
		
Consider an operator $A$ in the third form. So $A :=  \Pi_{j=1}^{k}{T^*}^{m_j}T^{n_j}$ for some $k\geq 1, m_j, n_j \geq 1$. We consider two cases: $n_1 < p$ and $n_1 \ge p$.\\
		\textit{Case 1}. Suppose $n_1 < p$. Choose $m>1$ for which $mp-n_1>0$ and $mp-m_1>0$. 
		Let $Y_1:= {T^*}^pT^{p-n_1}{T^*}^{mp-n_1}$, $Y_2:=T^{mp}{T^*}^{n_1}{T^*}^{mp-m_1}$ and $Y:=Y_1Y_2$. Then, by re-writing $A = {T^*}^{m_1}T^{n_1}X$ where $X=\Pi_{j=2}^{k}{T^*}^{m_j}T^{n_j}$ for $k\ge 2$ and $X=I$ for $k=1$. And using the fact that $(T^k{T^*}^k)({T^*}^lT^l)=({T^*}^lT^l)(T^k{T^*}^k)$ because the operators inside the parentheses are diagonal operators and so they commute, one obtains
\begin{align*}
		YA=Y_1Y_2A = Y_1Y_2{T^*}^{m_1}T^{n_1}X &= Y_1(T^{mp}{T^*}^{n_1}{T^*}^{mp-m_1}){T^*}^{m_1}T^{n_1}X \\
		&= Y_1T^{mp}({T^*}^{n_1}{T^*}^{mp})T^{n_1}X \\
		&= Y_1T^{mp}({T^*}^{mp}{T^*}^{n_1})T^{n_1}X \\
		&= Y_1(T^{mp}{T^*}^{mp})({T^*}^{n_1}T^{n_1})X\\
		&= Y_1({T^*}^{n_1}T^{n_1})(T^{mp}{T^*}^{mp})X\qquad \text{(both are diagonals and so commute)}\\
		&= Y_1({T^*}^{n_1}T^{mp})(T^{n_1}{T^*}^{mp})X\qquad \text{(as $T^{n_1}$ and $T^{mp}$ commute)}\\
		&= ({T^*}^pT^{p-n_1}{T^*}^{mp-n_1}({T^*}^{n_1}T^{mp}))T^{n_1}{T^*}^{mp}X\\
		&=({T^*}^pT^{p-n_1}{T^*}^{mp}T^{mp})T^{n_1}{T^*}^{mp}X\\
		&= {T^*}^{mp}X\qquad \text{(because ${T^*}^{mp}T^{mp}=I$ and also ${T^*}^pT^p=I$)}
		\end{align*}
		This shows that for some $Y\in \Sc(T,T^*)$ and $m\ge 1$, $YA={T^*}^{m}X$.
		
		Before proceeding further with Case 1 ($n_1< p$), we need to show that Case 2 ($n_1 \ge p$) reduces to the same form in Case 1 (that $YA= {T^*}^{m}X$ for some  $Y\in \Sc(T,T^*)$ and $m\ge 1$). \\
		\textit{Case 2.} Suppose $n_1 \geq p$. Then $n_1 = rp + s$ for some $r\geq 1, 0\leq s <p$. By left multiplying $B= {T^*}^{rp}$ and because ${T^*}^{rp}T^{rp}=I$ (since $T^p$ is an isometry) one obtains
		$$BA = {T^*}^{rp}{T^*}^{m_1}(T^{n_1})X={T^*}^{rp}{T^*}^{m_1}(T^{rp}T^s)X= {T^*}^{m_1}({T^*}^{rp}T^{rp})T^sX ={T^*}^{m_1}T^sX.$$
		So, $BA = {T^*}^{m_1}T^sX$ where $0\leq s<p$. For $s=0, BA = {T^*}^{m_1}X$, the desired form ($YA={T^*}^{m}X$ for some $Y\in \Sc(T,T^*)$ and $m\ge 1$). For $0< s<p$, $BA = {T^*}^{m_1}T^sX$. Choose $Y$ as in Case 1 (based on $BA$, in particular only on $m_1$ and $s$) to obtain $YBA = {T^*}^mX$ for some $m \geq 1$ and $Y\in \Sc(T, T^*)$. Hence for Case 2 we also have the Case 1 conclusion: $YA={T^*}^{m}X$ for some  $Y\in \Sc(T,T^*)$ and $m\ge 1$.
		
		Hence in both the cases, $YA = {T^*}^{m}(\Pi_{j=2}^{k}{T^*}^{m_j}T^{n_j})$ or $YA={T^*}^{m}$ depending on whether $k\ge 2$ or $k=1$. If $YA={T^*}^{m}$, then clearly  ${T^*}^{m}\in (A)_{\Sc(T, T^*)}$ and so as explained earlier, it follows from Equation (\ref{ideals equation}) that $$\Sc(T,T^*)=({T^*}^m)_{\Sc(T,T^*)}\subseteq(A)_{\Sc(T,T^*)}\subseteq\Sc(T,T^*),$$ hence $(A)_{\Sc(T,T^*)}=\Sc(T,T^*)$. For the case $YA = {T^*}^{m}(\Pi_{j=2}^{k}{T^*}^{m_j}T^{n_j})$, $YA =\Pi_{j=2}^{k}{T^*}^{m'_j}T^{n_j} $ where $m'_2 = m_2+m$ and $m'_j = m_j$ for $j\geq 3$. Setting $Y_1=Y$ and applying this same process to $Y_1A$ that we initially applied to $A$ obtains $Y_2 \in \Sc(T,T^*)$ for which $Y_2Y_1A={T^*}^m(\Pi_{j=3}^{k}{T^*}^{m_j}T^{n_j})$ for some $m\ge 1$. And continuing obtains $Y_k\cdots Y_1A={T^*}^m$ for some $Y_i$'s $\in \Sc(T, T^*)$ for $1\le i \le k$ and $m\ge 1$. Hence ${T^*}^{m}\in (A)_{\Sc(T, T^*)}$. And again from Equation (\ref{ideals equation}), $$\Sc(T,T^*)=({T^*}^m)_{\Sc(T,T^*)}\subseteq(A)_{\Sc(T,T^*)}\subseteq\Sc(T,T^*),$$ hence $(A)_{\Sc(T,T^*)}=\Sc(T,T^*)$. This completes the proof that for every $A\in \Sc(T,T^*)$, $(A)_{\Sc(T,T^*)}=\Sc(T,T^*)$ which clearly implies the simplicity of $\Sc(T,T^*)$.
	\end{proof}
	Now as promised in the first paragraph of this periodic subsection, using Proposition~\ref{p-isometry} and Theorem~\ref{first p-theorem} above, we directly obtain the following Theorem \ref{periodic SI characterization} SI characterization for the class of weighted shifts with weight sequence $0_k \oplus \{\alpha_i\}$ where $\{|\alpha_i|\}$ is a periodic sequence with strictly positive weights.
	\begin{theorem}\label{periodic SI characterization}
		Let $T$ be a weighted shift with weights $0_k\oplus \{\alpha_n\}$ where $\{|\alpha_n|\}$ is a $p$-periodic sequence of nonzero numbers. Then the following are equivalent.
		\begin{enumerate}[label=(\roman*)]
			\item $\Sc(T,T^*)$ is an SI semigroup.
			\item $T^p = 0_k \oplus U$ with $U$ an isometry.
			\item $\Sc(T,T^*)$ is simple.
		\end{enumerate}
	\end{theorem}
	As mentioned earlier after Example \ref{p-example}, we have obtained a characterization of SI semigroups $\Sc(T, T^*)$ generated by those weighted shifts with weight sequence $0_k\oplus \{\alpha_n\}$ where $\{|\alpha_n|\}$ is a $p$-periodic sequence of nonzero weights (and in particular, for those SI semigroups $\Sc(T, T^*)$ generated by a weighted shift $T$ with periodic nonzero weights). We next obtain a characterization of SI semigroups $\Sc(T, T^*)$ generated by those weighted shifts with weights $0_k\oplus \{\alpha_n\}$ where $\{|\alpha_n|\}$ is an eventually constant sequence of nonzero weights. (By an eventually constant weight sequence $\{\alpha_n\}$, we mean for some $N\geq1$ and $\alpha \in \mathbb{C}$ for which $\alpha_j = \alpha$ for all $j\geq N$.)  For that we need to prove simplicity of $\Sc(T, T^*)$ for a subclass of quasi-isometries in Proposition \ref{quasi-isometries}. Recall that in \cite[Remark 1.22(v)]{PW21} we proved that for $T$ an isometry, $\Sc(T, T^*)$ is always simple. Under the slightly weaker assumption that $T^*T = I$ on $\range T$ (equivalently, $(T^*T)T=T$), we prove next that $\Sc(T, T^*)$ is simple. This class of operators $T$ that satisfy $(T^*T)T=T$ belong to the class of quasi-isometries (i.e., ${T^*}^2T^2 = T^*T$) introduced by Patel \cite{Patel}. 	
	\begin{proposition}\label{quasi-isometries}
		For $T\in \mathcal{B}(\mathcal{H})$ where $T$ satisfies $(T^*T)T=T$, $\Sc(T,T^*)$ is simple. 
	\end{proposition}
	\begin{proof}
		Since $(T^*T)T=T$, by a straightforward induction argument one obtains 
		for all $n \geq 1$,
		\begin{equation}\label{equation1}
		{T^*}^nT^{n+1}=T
		\end{equation}
		Hence also, for all $n \geq 1$,
		\begin{equation}\label{equation2}
		{T^*}^{n+1}T^n=T^*
		\end{equation}
		Recall the semigroup list \cite[Proposition 1.6]{PW21}:\\
		\noindent $\Sc(T,T^*) =  \{T^n, {T^*}^n,  \Pi_{j=1}^{k}{T^*}^{m_j}T^{n_j},
		(\Pi_{j=1}^{k}{T^*}^{m_j}T^{n_j}){T^*}^{m_{k+1}},\Pi_{j=1}^{k}T^{n_j}{T^*}^{m_j},$  $(\Pi_{j=1}^{k}T^{n_j}{T^*}^{m_j})T^{n_{k+1}}\}$ where $n \ge 1,\,  k\ge1,\, n_j, m_j \ge 1\, \text{ for }\, 1 \le j \le k \text{ and } n_{k+1},m_{k+1}\geq 1$. To prove $\Sc(T,T^*)$ is simple, it suffices to show that the principal ideal generated by each form in the semigroup list coincides with the entire semigroup $\Sc(T,T^*)$. Furthermore, it suffices to show that the principal ideals generated by all the fourth and sixth forms coincide with $\Sc(T,T^*)$ because each principal ideal generated by each of the other forms contains a fourth and a sixth form.
		
		Consider an operator $A$ in the fourth form. So $A=(\Pi_{j=1}^{k}{T^*}^{m_j}T^{n_j}){T^*}^{m_{k+1}}$ for some $m_j, n_j \geq 1$ and $m_{k+1} 
		\geq 1$. Let $s = \sum_{j=1}^{k} n_j$ and $r = \sum_{j=1}^{k+1} m_j$. Then,
		
		\begin{equation*}
		\begin{aligned}
		{T^*}^{s}A &= {T^*}^s({T^*}^{m_1}T^{n_1})(\Pi_{j=2}^{k}{T^*}^{m_j}T^{n_j}){T^*}^{m_{k+1}}\\
		&= {T^*}^{s+ m_1 - n_1 -1} ({T^*}^{n_1 +1}T^{n_1})(\Pi_{j=2}^{k}{T^*}^{m_j}T^{n_j}){T^*}^{m_{k+1}} \quad \text{(add and substract $n_1+1$ from the power $s$ of $T^*$)}\\ 
		&= {T^*}^{s+ m_1 - n_1}(\Pi_{j=2}^{k}{T^*}^{m_j}T^{n_j}){T^*}^{m_{k+1}} \qquad \qquad (\text{from Equation (\ref{equation2}) above ${T^*}^{n_1 +1}T^{n_1} = T^*$})\\
		&= {T^*}^{s+(m_1-n_1)+(m_2 -n_2 -1)}({T^*}^{n_2 +1}T^{n_2})(\Pi_{j=3}^{k}{T^*}^{m_j}T^{n_j}){T^*}^{m_{k+1}}\\
		&={T^*}^{s+(m_1-n_1)+(m_2-n_2)}(\Pi_{j=3}^{k}{T^*}^{m_j}T^{n_j}){T^*}^{m_{k+1}} \qquad (\text{again from Equation (\ref{equation2})})\\
		\qquad \vdots\\
		&= {T^*}^{\sum_{j=1}^{k} m_j}{T^*}^{m_{k+1}}\\
		&= {T^*}^r \qquad (\text{from Equation (\ref{equation1}) above})
		\end{aligned}
		\end{equation*}
		Since ${T^*}^{s+1}AT^r \in (A)_{\Sc(T, T^*)}$ and ${T^*}^{s+1}AT^r=T^*({T^*}^sA)T^r={T^*}^{r+1}T^r=T^*$ (from Equation \ref{equation2}), one obtains $T^* \in (A)_{\Sc(T, T^*)}$. Also note that $({T^*}^{s}A)T^{r+1} ={T^*}^rT^{r+1}= T$ (from Equation (\ref{equation1})), so $T\in (A)_{\Sc(T, T^*)}$. And since $T,T^*\in(A)_{\Sc(T, T^*)}$, $(A)_{\Sc(T, T^*)} = \Sc(T, T^*)$. 
		
		We next consider the sixth form. So $A= (\Pi_{j=1}^{k}T^{n_j}{T^*}^{m_j})T^{n_{k+1}}$ for some $n_j, m_j \geq 1$, $1\leq j\leq k$, and $n_{k+1} 
		\geq 1$. The operator ${T^*}^{n_1}A{T^*}^{n_{k+1}} \in (A)_{\Sc(T, T^*)}$. Note that ${T^*}^{n_1}A{T^*}^{n_{k+1}}$ is back in the fourth form. Hence $({T^*}^{n_1}A{T^*}^{n_{k+1}})_{\Sc(T, T^*)}=\Sc(T, T^*)$. But $({T^*}^{n_1}A{T^*}^{n_{k+1}})_{\Sc(T, T^*)}\subset (A)_{\Sc(T, T^*)}$ so $(A)_{\Sc(T, T^*)}=\Sc(T, T^*)$.
	\end{proof}
	
	Now we can give a characterization of SI semigroups $\Sc(T, T^*)$ generated by those weighted shifts with weights $0_k\oplus \{\alpha_n\}$ where $\{|\alpha_n|\}$ is an eventually constant sequence of nonzero weights.
	Early on we noticed that the weighted shift operator $T$ with weight sequence $(a, 1, 1, \dots)$ where $a\in \mathbb{C}$ are examples of quasi-isometries which further statisfy $(T^*T)T=T$. So we first studied the impact of the SI property for $\Sc(T, T^*)$ on this subclass of quasi-isometries but found the stronger condition of simplicity in Proposition \ref{quasi-isometries}.

	In Theorem \ref{THEOREM}, we provided a necessary condition on the weight sequence of $T$ for $\Sc(T, T^*)$ to be SI. And Example \ref{p-example} showed that that necessary condition is not sufficient.
	But in this rather restrictive class of weighted shifts, we obtain in the following corollary, a necessary and sufficient condition for $\Sc(T, T^*)$ to be SI, as mentioned in the first paragraph of this periodic subsection.
	
	\begin{corollary}\label{eventually-constant}
		Let $T$ be a weighted shift with weights $0_k\oplus \{\alpha_n\}$ where $\{|\alpha_n|\}$ is an eventually constant sequence of nonzero weights. Then $\Sc(T, T^*)$ is an SI semigroup if and only if $\{|\alpha_n|\}$ has the form  $(a,1, 1, \dots)$. Moreover in this situation, $\Sc(T, T^*)$ is simple.
	\end{corollary}
	\begin{proof}
		Similar to the discussion in the beginning of the proof of Theorem~\ref{first p-theorem}, we can write $T=0_k \oplus T_1$ where $T_1$ is a weighted shift with nonzero weights $\{\alpha_n\}$. Then the SI property of $\Sc(T, T^*)$ is equivalent to the SI property of $\Sc(T_1,T^*_1)$. Furthermore, $T_1$ is unitarily equivalent to a weighted shift $S$ with strictly positive weights $\{|\alpha_n|\}$. And SI being a unitarily invariant property, so $\Sc(T_1,T^*_1)$ is SI if and only if $\Sc(S,S^*)$ is SI. Therefore without loss of generality we can assume that $T$ has strictly positive eventually constant weight sequence $\{\alpha_n\}$. 
		
		Given that $T$ has an eventually constant weight sequence $\{\alpha_n\}$, then there is an $\alpha >0$ (because all $\alpha_n$ are strictly positive) and a least $m\ge 1$ for which $\alpha_n=\alpha$ for $n \ge m$ (and $\alpha_{m-1}\ne \alpha$ when $m\ge 2$) . Suppose $\Sc(T,T^*)$ is an SI semigroup. First we will show that $\alpha=1$ irrespective of the value of $m$ and then we will show that $m\le 2$. Both these together will yield the required form for the weight sequence. Let $r:=\max\{m,2\}$, then by Theorem~\ref{THEOREM}, $\alpha_r\beta_r=1$, where $\beta_r$ is a product of powers of $\alpha_i$'s with the index $i\ge r \ge 2.$ Also since $r \ge m,~ \alpha_i=\alpha$ for all $i\ge r$, therefore $\alpha_r \beta_r=1$ implies that $\alpha^k=1$ for some $k \ge 2$, which further implies that $\alpha=1$. Hence we obtain $\alpha_n =1$ for all $n \ge m.$ Next we prove that $m \le 2$. Suppose not. Then $m \ge 3$ and so $\alpha_{m-1}\ne 1$. The latter is not possible because, again by Theorem~\ref{THEOREM}, $1/{\alpha_{m-1}}$ is a product of some $\alpha_i$'s with $i\ge m-1\ge 2$. Therefore $1/{\alpha_{m-1}}=\alpha_{m-1}^s$ for some $s\ge 1$, so $\alpha_{m-1}=1$ against $\alpha_{m-1}\ne1$.
		
		This completes the proof for one direction of the result, that is, $\Sc(T,T^*)$ SI implies that the weight sequence must be of the form  $(a,1, 1, \dots)$. And when the weight sequence has this form one can easily verify that $T$ satisfies $(T^*T)T=T$. Then by Proposition \ref{quasi-isometries} $\Sc(T, T^*)$ is simple and hence SI. 
	\end{proof}

Finally in this periodic subsection, we investigate the larger class of weight sequences with their absolute values almost periodic (see Definition \ref{definition almost periodic} below), which properly contains the classes of weight sequences we considered in Theorem \ref{first p-theorem} and Corollary \ref{eventually-constant} and also contains all the eventually periodic weight sequences. For the weighted shifts $T$ with weight sequence in this class, we generalize Theorem \ref{first p-theorem} by obtaining a necessary condition in Theorem \ref{almost periodic theorem} below for the SI property of the semigroup $\Sc(T, T^*)$. We first need the following definition and observations:

We define the complex analogue of \cite[Definition]{Ridge} as follows:
\begin{definition}\label{definition almost periodic}
	A sequence $\{\alpha_n\}$ is \textit{almost $p$-periodic} if there is a periodic sequence $\{c_n\}$ with period $p$ for which $\lim_{n}(\alpha_n-c_n)=0$. If $\{c_n\}$ has period $p$, as in \cite[Definition]{Ridge}, we  define the \textit{periodic mean} to be $q:=|c_1c_2\cdots c_p|^{1/p}$.
\end{definition}
That is, we define the periodic mean of an almost periodic sequence $\{\alpha_n\}$ to be the periodic mean of its approximating sequence $\{c_n\}$. And it is easy to show that this definition is well-defined, that is, if $\{c'_n\}$ is another approximating periodic sequence, then clearly $\{c_n\}=\{c'_n\}$.

\begin{remark}\label{remark almost p}
	Observe that almost periodic sequences are automatically bounded and their periodic means are independent of the choice of period $p$ (see paragraph after Definition \ref{first p-mean definition}). Also, clearly the product of any two almost $p$-periodic sequences is almost $p$-periodic. Indeed, if $\{\alpha_n\}$ and $\{\beta_n\}$ are any two almost $p$-periodic sequences with their respective approximating sequences $\{c_n\}$ and $\{b_n\}$, then the product sequence $\{\alpha_n\beta_n\}$ is approximated by the sequence $\{c_nb_n\}$. Therefore, if $\{\alpha_n\}$ and $\{\beta_n\}$ have their respective periodic means $q_1$ and $q_2$, then $\{\alpha_n\beta_n\}$ has periodic mean $q_1q_2$ (because the periodic mean of the approximating sequence $\{c_nb_n\}$ is $q_1q_2$, see discussion prior to Corollary \ref{periodic mean of products}).
	
	Once $\{\alpha_n\}$ is almost $p$-periodic with periodic mean $q$, so also are its tail sequences ($\{\alpha_{n+l}\}_{l\ge 1}$) and the sequences $0_k\oplus \{\alpha_n\}$ almost $p$-periodic with periodic mean $q$.
\end{remark}



From this Remark \ref{remark almost p} and Proposition \ref{multiplication} we obtain:
\begin{corollary}\label{almost periodic corollary}
	For any $D^{(\alpha)}_{k},D^{(\beta)}_{l}\in \mathfrak{D}$ with almost $p$-periodic sequences $\alpha, \beta$ with periodic means $q_1$, $q_2$, respectively, the product diagonal $D^{(\alpha)}_{k}D^{(\beta)}_{l} = D^{(\gamma)}_{k+l}$ has sequence $\gamma$ almost $p$-periodic with periodic mean $q_1q_2$.  
\end{corollary}
Now we give a necessary condition in terms of the periodic mean for $\Sc(T,T^*)$ to be an SI semigroup when generated by a weighted shift with  absolute values of its weight sequence almost periodic. 
\begin{theorem}\label{almost periodic theorem}
	Let $T$ be a nonzero weighted shift with weights $\{\alpha_n\}$ where $\{|\alpha_n|\}$ is an almost periodic sequence with periodic mean $q$. If $\Sc(T,T^*)$ is an SI semigroup, then either $q=1$ or $q=0$.
\end{theorem}
Note that the proof of this Theorem \ref{almost periodic theorem} is similar to the proof of Theorem \ref{first p-theorem} except that $q=0$ could be a possibility.

The converse of above Theorem \ref{almost periodic theorem} does not hold, in general, see Example \ref{almost first example} for $q=1$ and Example \ref{almost second example} for $q=0$ below:
\begin{example}\label{almost first example}
	The weighted shifts $T$ with weight sequences $\{1-1/n\}$ and $\{1+1/n\}$ are almost periodic where the approximating periodic sequence for both is the constant sequence $1$ with periodic mean $1$. Yet we discussed earlier as well in the paragraph before Example \ref{p-example} that for both these weight sequences, the necessary reciprocal condition in Theorem \ref{THEOREM} fails, so $\Sc(T,T^*)$ are not SI.
\end{example}
\begin{example}\label{almost second example}
	The weighted shift $T$ with weight sequence $\{\alpha_n\}$ where $$\alpha_n:=\begin{cases}
	1/n &\quad \text{ for } n~\text{odd}\\
	1 & \quad \text{ for } n~\text{even}
	\end{cases}$$ is almost periodic with the approximating periodic sequence $(0,1,0,1,\ldots)$, which has periodic mean $0$. Again clearly for $\{\alpha_n\}$ the necessary reciprocal condition in  Theorem \ref{THEOREM} fails, so $\Sc(T,T^*)$ is not SI.
\end{example}
\begin{remark}\label{almost p-spectrum}
	(Impact of SI on the spectrum): Note that for the SI semigroups $\Sc(T,T^*)$ generated by weighted shifts $T$ with their absolute value weight sequences almost periodic, the spectrum of $T$ is contained in $\overline{\mathbb{D}}$, indeed even more, $\sigma(T)=\{0\}$ or $\sigma(T)=\overline{\mathbb{D}}$. That is, for $T$  a weighted shift with weights $\{\alpha_n\}$ where $\{|\alpha_n|\}$ is an almost periodic sequence with periodic mean $q$, it follows from \cite[Theorem 2]{Ridge} that $\sigma(T)=\{z\in \mathbb{C}|~ |z|\le q\}$. Also since $\Sc(T,T^*)$ is SI, by Theorem \ref{almost periodic theorem}, $q=0$ or $q=1$.  Hence, $\sigma(T)=\{0\}$ or $\sigma(T)=\overline{\mathbb{D}}$.
	
	In particular, for the subclasses of weighted shift operators $T$ in Theorem \ref{periodic SI characterization} and Corollary \ref{eventually-constant} (with the weight sequences $0_k\oplus \{\alpha_n\}$ where $\{|\alpha_n|\}$ is a $p$-periodic sequence of nonzero numbers and $\{|\alpha_n|\}$ is an eventually constant sequence of nonzero numbers), under the assumption of the SI property for $\Sc(T,T^*)$, we obtain $\sigma(T)=\overline{\mathbb{D}}$. Because for both these subclasses of weighted shift operators $T$, the SI property for $\Sc(T,T^*)$ implies that the periodic mean $q$ of the absolute value weight sequence is $1$ (which follows from Theorem \ref{first p-theorem} and Corollary \ref{eventually-constant}).
\end{remark}

\section{The impact of the SI property of $\Sc(T,T^*)$ on the spectral density for $T$: hyponormal, essentially normal, and weighted shifts} 

Earlier we obtained a complete characterization of SI semigroups $\Sc(T,T^*)$ for normal operators $T$ in \cite[Remark 1.13 and Theorem 2.1]{PW21}. In this section, we consider the broader class of hyponormal operators for our SI investigation of $\Sc(T,T^*)$ (note the proper inclusions: normal operators $\subset$ subnormal operators $\subset$ hyponormal operators \cite[Remark]{Stamp65}); and generalize our study of SI semigroups to unital $C^*$-algebras in Subsection 4.1. As a consequence, under the SI assumption, we obtain nontrivial projections in singly generated unital $C^*$-algebra generated by a non-invertible normal element (Corollary \ref{projection}).

We also found deep connections between the study of the SI property of $\Sc(T, T^*)$ and the spectral density of $T$ for hyponormal operators and essentially normal operators (i.e., $\pi(T)$ normal in the Calkin algebra, equivalently, with compact self commutator $T^*T-TT^*$) (Proposition \ref{hyponormal}, Corollary \ref{corollary6.4}, and Theorem \ref{fullspec}).



The most prominent example of a hyponormal operator is the infinite unilateral shift. Note that the approximate point spectrum of the unilateral shift is the unit circle $S^1$ (see \cite[Solution 82]{Hal82}). Interestingly, it turns out that for any hyponormal operator $T$, under the assumption of the SI property of $\Sc(T, T^*)$, the approximate point spectrum of $T$ is a subset of $S^1 \cup \{0\}$ which is proved in the following lemma. We use $\sigma_{ap}(T)$ to denote the approximate point spectrum of $T$.

\begin{lemma}\label{lemma6.1}
	For $T\in B(\mathcal{H})$ a nonselfadjoint hyponormal operator, if $\Sc(T,T^*)$ is an SI semigroup, then $${\{0\} \ne~} \sigma_{ap}(T)\subset S^1\cup\{0\}.$$
\end{lemma}
\begin{proof}
Since $T$ is a nonselfadjoint operator, $T \neq 0$. We first claim $\sigma_{ap}(T)$ contains a nonzero value.
For $T$ a nonzero hyponormal operator, by \cite[Theorem 1]{Stamp62}, its spectral radius $r(T) = ||T|| > 0$. Therefore because the spectrum $\sigma(T)$ is compact and hence closed, it contains a nonzero boundary point with modulus the spectral radius. And since its boundary
$\partial(\sigma(T)) \subset \sigma_{ap}(T)$ \cite[Problem 78]{Hal82}, one has the claim.

Since $\Sc(T,T^*)$ is an SI semigroup, the principal ideal $(T)_{\Sc(T,T^*)}$ is a selfadjoint ideal. Therefore, $T^*=XTY$ for some $X,Y\in \Sc(T,T^*)\cup\{I\}$, where $X,Y$ cannot both be the identity operator, otherwise $T$ would be selfadjoint. Since $T$ is hyponormal, by \cite[Corollary 10]{JB70}, for each $0 \neq \lambda \in \sigma_{ap}(T)$, there exists a character $\phi$ on the unital $C^*$ algebra generated by $T$,  $C^*(T)$, such that $\phi(T) =\lambda$. Since characters are multiplicative $*$-preserving linear functionals, and $X,Y$ when not the identity are words in $T, T^*$, by applying $\phi$ to $T^*=XTY$, we obtain 
$$\overline{\phi(T)} = \phi(X)\phi(T)\phi(Y) = \phi(T)^n\overline{\phi(T)^m},$$ 
where $n \ge 1$, but additionally $n,m$ must satisfy, for some $m=0, n\geq 2$ or $n\geq 1, m\geq 1$. In the former case, for $m=0$, one has $n\geq 2$ because both $X$ and $Y$ cannot be the identity operator and $T$ is not selfadjoint. And in the latter case, if $n=1$, $m\ne 0$ again since $T$ is nonselfadjoint. 
Therefore, since $\phi(T) =\lambda \neq 0$, taking absolute values, one obtains
$$|\lambda|=|\lambda|^{n+m}$$ where $n+m\geq 2$. This implies that $0 \neq \lambda\in S^1$. Hence $\sigma_{ap}(T) \subset S^1\cup\{0\}$.
\end{proof}

\begin{remark}\label{R6.1}
(i). We note that under the hypothesis of Lemma \ref{lemma6.1}, $\sigma(T) \subset \overline{\mathbb D}$. Indeed, combining the general fact that $\partial(\sigma(T)) \subset \sigma_{ap}(T)$ \cite[Problem 78]{Hal82} and the inclusion in Lemma  \ref{lemma6.1}, one has $\sigma_{ap}(T) \subset S^1 \cup \{0\}$ and hence $\partial(\sigma(T))\subset S^1 \cup \{0\}$. Moreover, $\partial(\sigma(T))\neq \{0\}$, otherwise $r(T) = ||T||=0$ implying $T=0$ which contradicts $T\neq 0$. Therefore, $\emptyset \ne \partial(\sigma(T)) \setminus \{0\} \subset S^1$. Hence, $r(T)=1$ (so $||T|| =1$) and because the spectral radius is 1, one concludes that $\sigma(T) \subset \overline{\mathbb D}$.  \\
(ii) Lemma \ref{lemma6.1} implies that the planar area measure of $\sigma_{ap}(T) = 0$, if $\Sc(T, T^*)$ is SI.\\
(iii) The converse of Lemma \ref{lemma6.1} does not hold, in general. For example, consider the 
nonzero hyponormal weighted shift $T$ with weights 
$\{\alpha_n\}=\{1-1/(n+1)\}$. 
Since $0 \neq \alpha_n\to 1$, it follows that $T$ is injective allowing us to apply \cite[Corollary 1]{Ridge} to conclude that  $\sigma(T)=\overline{ \mathbb{D}}$  and $\sigma_{ap}(T) \subset S^1$. 
And as proved in the first paragraph of the proof of Lemma \ref{lemma6.1}, nonzero hyponormal operators have nonzero approximate point spectra, so we have $\{0\} \ne \sigma_{ap}(T) \subset S^1 \cup \{0\}$. 
However $\Sc(T,T^*)$ is a non-SI semigroup. Indeed, if $\Sc(T, T^*)$ were an SI semigroup, then $\alpha_2 = 2/3$ must have  its inverse as the product of certain scalars $\alpha_j$'s (possibly including repetition) where $j\geq 2$ by Theorem \ref{THEOREM}. But the product of any powers of $\alpha_j$'s for $j \geq 2$ is a number strictly less than $1$ so it cannot be the inverse of $2/3$. Therefore $\Sc(T,T^*)$ is not an SI semigroup.
\end{remark}

Note that every normal operator is a hyponormal operator so using Lemma \ref{lemma6.1}, we provide an alternate short proof of \cite[Theorem 2.1]{PW21} in the corollary below.

\begin{corollary}\label{normal}
\cite[Theorem 2.1]{PW21} If $T$ is a nonselfadjoint normal operator and $\Sc(T, T^*)$ is an SI semigroup, then $T$ is unitarily equivalent to $U \oplus 0$ for some unitary operator $U$ (the zero summand may be absent). 
\end{corollary}
\begin{proof} $T$ is nonselfadjoint so $T \ne 0$. Since a normal operator is also a hyponormal operator, so by Lemma \ref{lemma6.1}, $\sigma_{ap}(T) \subset S^1\cup\{0\}$. Moreover, for a normal operator, $\sigma(T) = \sigma_{ap}(T)$ \cite[Problem 79]{Hal82}. It then follows that $\sigma(T) \subset S^1\cup\{0\}$. If $T$ is invertible, then $\sigma(T) \subset S^1$. This implies that $T$ is a unitary operator \cite[Chapter 1, Section 1.3]{Olsen}. If $T$ is not invertible, then $0$ is an isolated point in $\sigma(T)$. And hence, $0$ is an eigenvalue of $T$ by \cite[Theorem 2]{Stamp62} and the eigenspace corresponding to the eigenvalue $0$ is a reducing subspace for $T$ \cite[Lemma 6]{Stamp62}. Therefore, $T$ is unitarily equivalent to $U \oplus 0$ where $U$ is a unitary operator on $(\ker T)^\perp$ with respect to the decomposition of the Hilbert space $\mathcal{H}=(\ker T)^\perp \oplus\ker T$.
\end{proof}

\textbf{Hyponormality versus normality.}  We are now ready to explore for hyponormality, the relationship between the SI property for $\Sc(T,T^*)$ and the  spectrum of $T$ in terms of spectral density as described in the Introduction concerning Section 4. 
The next theorem asserts that in addition to the hyponormality of $T$ and the SI property of $\Sc(T, T^*)$, if we further assume that the boundary of the spectrum of $T$ excludes at least one point of the unit circle, then we have normality of $T$. 

\begin{theorem}\label{hypo}
Suppose $T\in B(\mathcal{H})$ is a hyponormal operator and boundary of the spectrum of $T$ excludes at least one point of the unit circle. If $\Sc(T,T^*)$ is an SI semigroup, then $T$ is normal.
\end{theorem}
\begin{proof} Since $\Sc(T, T^*)$ is an SI semigroup, every ideal in $\Sc(T, T^*)$ is selfadjoint. In particular, the principal ideal, $(T)_{\Sc(T, T^*)}$ is selfadjoint. So $T^* \in (T)_{\Sc(T, T^*)}$. Therefore, 
\begin{equation}\label {Eq}
T^* = XTY
\end{equation}
for some $X, Y \in \Sc(T, T^*) \cup \{I\}$. If $X = Y = I$, then $T^* = T$, hence the normality of $T$. So we may assume that $T$ is nonselfadjoint, in which case either $X \neq I$ or $Y \neq I$. To show that $T$ is normal, we will prove that $\sigma(T) \subset S^1 \cup \{0\}$ which has Lebesgue measure zero, and hence so also $\sigma(T)$, implying normality of $T$ by \cite[Corollary]{Put70}. 

Since $T$ is a nonselfadjoint hyponormal operator and $\Sc(T, T^*)$ is an SI semigroup, ${\{0\} \ne~} \sigma_{ap}(T)\subset S^1\cup\{0\}$ (Lemma \ref{lemma6.1}). To prove $\sigma(T) \subset S^1 \cup \{0\}$, we will show that $\sigma(T) = \partial(\sigma(T))$, the latter of which is a subset of $\sigma_{ap}(T)$ \cite[Problem 78]{Hal82}.   
We will first show that $\sigma(T) \subset \overline{\mathbb D}$. Since always $\partial (\sigma(T)) \subset \sigma_{ap}(T)$ as stated above, this implies that 
$\partial (\sigma(T)) \subset S^1 \cup \{0\}$. It follows that $\sigma(T) \subset \overline{\mathbb D}$. Indeed, if $\lambda \in \sigma(T)$ is a point outside $\overline{\mathbb D}$,  since $\sigma(T) = \partial (\sigma(T)) \cup \text{int}(\sigma(T))$ and $\partial (\sigma(T)) \subset S^1 \cup \{0\}$, then $\lambda$ lies in an open ball inside $\text{int}(\sigma(T))$ and outside $\overline{\mathbb D}$. Then the ray $\{t\lambda \mid t \geq 0\}$ from $0$ through $\lambda$ must exit the bounded set $\sigma(T)$ in a boundary point $t_0\lambda$ because $1< t_0:=\text{sup}\{t\geq 0~|~t\lambda\in\sigma(T)\}<\infty$
	(due to the compactness of the spectrum and that for $t=1,~t\lambda=\lambda\in \text{int}(\sigma(T))\setminus\overline{\mathbb{D}}$). Then it is clear that $t_0\lambda\in\partial(\sigma(T))$, contradicting the inclusion $\partial(\sigma(T))\subset \mathbb{S}^1\cup \{0\}$. This completes the proof that $\sigma(T) \subset \overline{\mathbb D}$.

From $\sigma(T) \subset \overline{\mathbb D}$ we claim further that $\sigma(T) \subset S^1 \cup \{0\}$. 
Since $\sigma(T) = \partial (\sigma(T)) \cup \text{int}(\sigma(T))$ and as shown above $\partial (\sigma(T)) \subset S^1 \cup \{0\}$, it suffices to show that $\text{int}(\sigma(T)) = \emptyset$.
Suppose otherwise that $\text{int}(\sigma(T)) \neq \emptyset$. Since $\sigma(T) \subset \overline{\mathbb D}$ and the interior is open, one can chose a $0 \neq z \in \text{int}(\sigma(T)) \subset \mathbb D$. And as by the hypothesis that boundary of the spectrum of $T$ excludes at least one point of the unit circle, choose a $w$ on the unit circle with $w \in \rho(T)$, the open resolvent set of $T$, and hence the resolvent also contains an open ball around $w$. 
So $z, w$ are respectively in the disjoint open sets $\text{int}(\sigma(T))$ and $\rho(T)$. Since $z\in \mathbb D$ and $w \in S^1$, the line segment $[z,w)$ lies entirely in $\mathbb D$ and by varying slightly either $z$ or $w$ we can assume further from a simple geometric argument that $[z,w)$ lies entirely in $\mathbb D \setminus \{0\}$. Then representing $[z,w) := \{(1-t)z + tw \mid 0 \leq t < 1\}$, set $t_o = \text{sup} \{0 \leq t < 1 \mid (1-t)z + tw \in \sigma(T)\}$. Because $\text{int}(\sigma(T))$ and  $\rho(T)$ are open and disjoint, it is clear that $0 \leq t_o < 1$ (as $z \in \text{int}(\sigma(T))$ and $w \in \rho(T)$ as well as the whole segment  $((1-t_o)z + t_ow,w] := \{(1-t)z + tw \mid t_o <t \le 1\} \subset \rho(T)$. It is clear by construction that $(1-t_o)z + t_ow$ is a boundary point of the spectrum and that this boundary point lies inside the line segment $[z,w)$, hence inside the open disk $\mathbb D \setminus \{0\}$. That is, $(1-t_o)z + t_ow$ is a boundary point of $\sigma(T)$ that lies in $\mathbb D \setminus \{0\}$, against what we showed earlier that the boundary of the spectrum lies entirely in $S^1 \cup \{0\}$. Thus we have showed that 
$\sigma(T) \subset S^1 \cup \{0\}$. Since the latter set has area zero, so $\sigma(T)$ has area zero and hence $T$ is normal, as mentioned above, by \cite[Corollary]{Put70}.
\end{proof}

\begin{remark}\label{R-shift}
(i). The assumption in Theorem \ref{hypo} on the boundary of the spectrum is necessary for the conclusion to hold because if we consider $T$ to be the unilateral shift, then $T$ is hyponormal as $T^*T - TT^*$ is a rank-one projection operator and $T$ being an isometry, $\Sc(T, T^*)$ is an SI semigroup \cite[Remark 1.22(v)]{PW21}. But $T$ is not normal. Yet $\sigma(T) = \overline{\mathbb D}$ whose boundary is the entire unit circle.

\noindent(ii). 	For $T$ a nonselfadjoint hyponormal operator, if  $\Sc(T,T^*)$ is an SI semigroup and $\sigma(T)$ excludes at least one point of $\mathbb{S}^1$, then by Theorem \ref{hypo}, $T$ is normal. And consequently by \cite[Theorem 2.1]{PW21}, $T$ is unitarily equivalent to $U\oplus0$ ($0$ may be absent) which is further equivalent to simplicity of  $\Sc(T,T^*)$.  
\end{remark}
Under the SI property of $\Sc(T, T^*)$ assumption, we prove next that the normality of a hyponormal operator is equivalent to $\area(\sigma(T)) = 0$; and (denoting the essential spectrum by $\sigma_e(T)$) the essential normality of a subnormal operator is equivalent to $\area(\sigma_e(T)) = 0$ (Proposition \ref{hyponormal} and Corollary \ref{corollary7.3} below).

\begin{proposition}\label{hyponormal}
	Let $T\in B(\mathcal{H})$ be a hyponormal operator. Suppose $\Sc(T,T^*)$ is an SI semigroup. Then, $T$ is normal if and only if Area$(\sigma(T))=0$.
\end{proposition}
\begin{proof}
	For $T$ a hyponormal operator, Putnam's Inequality \cite[Theorem 1]{Put70} is given by  
	$$\pi\,||T^*T-TT^*|| \leq \area(\sigma(T)).$$ 
Hence $\area(\sigma(T))=0$ implies that $T$ is normal. Conversely, let $T$ be a normal operator. 
 When $T$ is selfadjoint, its spectrum is on the line and so has zero area. And when $T$ is nonselfadjoint,  
 the SI property of $\Sc(T, T^*)$ implies that $T$ is unitarily equivalent to $U\oplus 0$ (the zero summand may be absent) with $U$ a unitary operator by \cite[Theorem 2.1]{PW21}.	
	 Therefore, $\sigma(T)\subset S^1\cup\{0\}$, and hence Area$(\sigma(T))=0$.
\end{proof}
Without the SI property of $\Sc(T, T^*)$ assumption for a hyponormal operator, the normality of $T$ does not imply Area$(\sigma(T))=0$ as shown in the following example.
 \begin{example}\label{ee-1}
	Consider the measure space $(X,\mu)$, where $X$ is the closed unit disk centered at the origin in the complex plane and $\mu$ is the Lebesgue measure on $X$. Then the multiplication operator $M_z : L^2(X,\mu)\longrightarrow L^2(X,\mu)$ is a normal nonselfadjoint operator, but $\Sc(M_{z}, M_{\bar{z}})$ is not SI by \cite[Theorem 2.1]{PW21} because $M_{z}$ is not a unitary operator (as $M_{\phi}$ is unitary if and only if $|\phi| = 1$ a.e.), nor is it of the form $U \oplus 0$ (compare their spectra). But the spectrum $\sigma(M_z)$ is the essential range of $\phi(z) = z$, which is equal to the closed unit disk as it is a continuous function, and hence has nonzero area. 
\end{example}

So far we have observed that for special classes of hyponormal operators, for example, in the case of normal operators and in the case of nonselfadjoint hyponormal operators for which the boundary of the spectrum excludes at least one point of the unit circle, the SI property of the semigroup $\Sc(T, T^*)$ yields its simplicity (see \cite[Theorem 2.1]{PW21} and Remark \ref{R-shift} (ii)). In light of this, it is of interest to us to ask the following question.

\begin{question} Does there exist $T\in B(\mathcal{H})$, a hyponormal non-normal operator such that $\Sc(T,T^*)$ is a non-simple SI semigroup?
\end{question}

\subsection{SI Semigroups in a unital $C^*$-algebra}
  In \cite{PW21} we studied the SI selfadjoint semigroups $\Sc(T,T^*)$ in the unital $C^*$-algebra $B(\mathcal{H})$. 
We generalize this notion of SI semigroup in a natural way to an arbitrary $*$-algebra, in particular, to an arbitrary $C^*$-algebra. 

\begin{definition}
	A $*$-algebra is an algebra $\mathcal{A}$ together with an involution map
	$$*: \mathcal{A}\rightarrow\mathcal{A}$$ 
	defined by $$a\longmapsto a^*$$ where $*$ is a conjugate-linear map such that $a^{**}=a$ and $(ab)^*=b^*a^*$ for all $a,b\in \mathcal{A}.$
	\end{definition}
\begin{definition} A $C^*$-algebra is a $*$-algebra $\mathcal{A}$ together with a submultiplicative norm such that $||a^*a||=||a||^2$ for all $a\in \mathcal{A}$ and $\mathcal A$ is complete with respect to that norm. Furthermore, if $\mathcal{A}$ has a unit, then we call $\mathcal{A}$ a unital $C^*$-algebra.
\end{definition}

The definitions of a semigroup, selfadjoint semigroup, and SI semigroup are easily generalized to a $*$-algebra $\mathcal A$ as mentioned below.	
\begin{definition}
	A semigroup $\mathcal{S}$ in $\mathcal{A}$ is a subset closed under multiplication. A selfadjoint semigroup $\mathcal{S}$ is a semigroup that is also closed under the involution map, i.e., $\mathcal{S}^*:=\{a^*| a\in \mathcal{S}\}\subset \mathcal{S}.$
	\end{definition}
\begin{definition} An ideal $J$ of a semigroup $\mathcal{S}$ in $\mathcal{A}$ is a subset of $\mathcal{S}$ closed under products of elements in $\mathcal{S}$ and $J$, i.e., $xa,ay\in J~\mbox{for}~ a\in J~\mbox{and}~x,y\in \mathcal{S}.$ And so also $xay\in J.$
\end{definition}
\begin{definition} A selfadjoint-ideal (SI) semigroup $\mathcal{S}$ in $\mathcal{A}$ is a semigroup for which every ideal $J$ of $\mathcal{S}$ is closed under involution, i.e., $J^*:=\{a^*|a\in J\}\subset J.$ 
\end{definition}

For a unital $C^*$-algebra $\mathcal A$ and $a \in \mathcal A$, consider the singly generated selfadjoint semigroup $\Sc(a, a^*)$ generated by $a$. Then, note that $\Sc(a, a^*) \subset C^*(a)$, the singly generated unital $C^*$-algebra generated by $a$.  A complete description of elements of $\Sc(a, a^*)$ can be obtained exactly similar to that of $\Sc(T, T^*)$, described prior to section 2, just by replacing $T$ with $a$. Precisely,

\noindent $ S(a, a^*) = \{a^n, {a^*}^n, \Pi_{j=1}^{k}a^{n_j}{a^*}^{m_j},  (\Pi_{j=1}^{k}a^{n_j}{a^*}^{m_j})a^{n_{k+1}}, \Pi_{j=1}^{k}{a^*}^{m_j}a^{n_j}, (\Pi_{j=1}^{k}{a^*}^{m_j}a^{n_j}){a^*}^{m_{k+1}},$ where $n \ge 1,\,  k\ge1,\, n_j, m_j \ge 1\, \text{for}\, 1 \le j \le k, ~\text{and}~n_{k+1}, m_{k+1} \geq 1\}$.

We begin with showing that $\Sc(a,a^*)$ possessing the SI property, for a non-invertible normal element $a \in \mathcal{A}$, we obtain nontrivial projections in $C^*(a)$ (Corollary \ref{projection}). Here we note that for a normal element $a \in \mathcal A$, $C^*(a)$ has no nontrivial projections if and only if $\sigma(a)$ (the spectrum of $a$) is connected. This follows from \cite[Theorem 2.1.13]{Murphy} which says that there exists a unique isometric $*$-algebra isomorphism $\phi$ from the $C^*$-algebra of all complex-valued continuous functions on $\sigma(a)$ onto $C^*(a)$. And a direct calculation shows that in the $C^*$-algebra of all complex-valued continuous functions on $\sigma(a)$, there are no nontrivial projections if and only if $\sigma(a)$ is a connected set. In Corollary \ref{projection}, we prove that under the SI property of $\Sc(a, a^*)$ for a normal non-invertible element $a$, the spectrum of $a$ is disconnected; thereby implying the existence of nontrivial projections in $C^*(a)$. Also, for a non-normal idempotent element $a$ in a $C^*$-algebra $\mathcal A$, the SI property of $\Sc(a, a^*)$ implies the existence of nontrivial projections in the singly generated $C^*$-algebra, $C^*(a)$ (see Remark~\ref{nontrivial_projection})
which is a consequence of Theorem~\ref{idempotent} where we proved that for a non-normal idempotent element $a \in \mathcal{A}$, the SI property of $\Sc(a,a^*)$ is equivalent to $a$ being a partial isometry. (For more general $C^*$-algebras, certain necessary and sufficient conditions for a $C^*$-algebra to be projectionless are stated in \cite[Proposition 3.3]{Black80}.)

Towards proving the existence of nontrivial projections in $C^*(a)$ for a non-invertible normal element $a$ (Corollary \ref{projection}), we first prove the following theorem.
\begin{theorem}\label{theorem6.3}
	For a normal nonselfadjoint element $a\in \mathcal{A}$, a unital C*-algebra, if $\Sc(a,a^*)$ is an SI semigroup, then $\sigma(a)\subset S^1\cup \{0\}$.
\end{theorem}
\begin{proof}
	Since $a$ is a normal element in a unital C*-algebra $\mathcal{A}$, it follows from \cite[Theorem 2.1.13]{Murphy} that there exists a unique isometric $*$-algebra isomorphism $\phi$ from the $C^*$-algebra of all complex-valued continuous functions on $\sigma(a)$ onto $C^*(a)$, where $C^*(a)$ is the C*-algebra generated by $1$ and $a$ for which $\phi(f)=a$ where $f$ is the identity function on $\sigma(a)$, i.e, $f(z)=z~\mbox{for all}~z\in\sigma(a)$. 
 Since $\phi$ is a $*$-isomorphism, $\phi(\bar{f})=a^*$. Moreover, $a\neq a^*$ because $a$ is a nonselfadjoint element.
	
	Supposing $\Sc(a,a^*)$ is an SI semigroup, then every ideal of $\Sc(a,a^*)$ is selfadjoint. In particular, the principal ideal $(a)_{\Sc(a,a^*)}$ is selfadjoint. Therefore, $a^*=xay$ for some $x,y\in \Sc(a,a^*)\cup\{1\}$ but where $x$ and $y$ are not both equal to $1$, otherwise $a$ would be selfadjoint against the hypothesized nonselfadjointness of $a$. Since $\phi$ is a $*$-preserving isometric isomorphism, 
so also is $\phi^{-1}$. So $a^*=xay$ implies that $$\phi^{-1}(a^*)=\phi^{-1}(x)\phi^{-1}(a)\phi^{-1}(y),$$ that is, besides being the identity function on the spectrum of $a$, $f$ must also satisfy 
$$\bar{f}=f^n\bar{f}^m, \text{~for~} n \ge 1, m \ge 0$$ but additionally where either $n> 1, m= 0$ or $m\geq 1, n\geq 1$. The case $m=0,n=1$ does not occur because in that case $\bar{f}=f$ implying $a^*=a$ against the nonselfadjointness of $a$. Furthermore, since $f(z) = z$ for $z\in \sigma(a)$, evaluating the function equation in the above display at $z$ and then taking the absolute value, one obtains:  for some $k \ge 2$,
$$|z|=|z|^k\hspace{.25cm}\mbox{for all}~z\in \sigma(a).$$ 
Therefore for each $z\in \sigma(a)$, either $z= 0$ or $|z|=1$. 
Hence $\sigma(a)\subset S^1\cup\{0\}$.
\end{proof}	
\begin{corollary}\label{projection}
For $a\in \mathcal{A}$ a normal nonselfadjoint \textit{non-invertible} element in a unital C*-algebra $\mathcal{A}$, if $\Sc(a,a^*)$ is an SI semigroup, then $C^*(a)$ has nontrivial projections.\end{corollary}	

\begin{proof} From Theorem \ref{theorem6.3}, $\sigma(a)\subset S^1\cup \{0\}$. 
Since $a$ is not invertible, $0 \in \sigma(a)$.  
Moreover,  $\sigma(a) \neq \{0\}$. Indeed, otherwise the spectral radius $r(a) = 0$. But since $a$ is normal, $r(a) = ||a||$ \cite[Lemma 1.2.7]{Lin} and since $a$ is also nonzero, $r(a) = ||a|| > 0$, a contradiction.

 Let $A=\sigma(a)\cap S^1$ and $B=\{0\}$. Since $\sigma(a)\subset S^1\cup \{0\}$ and $\sigma(a) \neq \{0\}$, the set $A\neq \emptyset$. Moreover, $A$ and $B$ are disjoint compact sets and $\sigma(a)=A\cup B$. Therefore, the characteristic functions $\chi_A$ and $\chi_B$ are continuous on $\sigma(a)$. 
Moreover, $\chi_A$  satisfies the equations $\chi^2_A=\chi_A$ and $\chi^*_A=\chi_A$, so also $\chi_B$. Hence, $\chi_A$ and $\chi_B$ are projection functions. Since $A$ and $B$ are non-empty proper subsets of $\sigma(a)$, so $\chi_A$ and $\chi_B$ are neither equal to the $0$ function nor equal to the constant function $1$. And hence $\chi_A$ and $\chi_B$ are nontrivial projection functions in $C(\sigma(a))$. 
Since $\phi$ from the proof of Theorem \ref{theorem6.3} is a $*$-isomorphism, $p:=\phi(\chi_A)$ and $q:=\phi(\chi_B)$ are nontrivial projections in $C^*(a)$. This completes the proof of the theorem. However, we can say more.

 Additionally, since $\chi_A+\chi_B=1$, $p+q=1$ where $1$ is the unit element of $C^*(a)$. Since $C(\sigma(a))$ is abelian, denoting the identity function on the spectrum by $f(z) = z$ and using the fact from Theorem \ref{theorem6.3} that $\phi(f) = a$, one has  
$$\bar{f}f\chi_A=\chi_A=f\bar{f}\chi_A.$$
	Applying $\phi$ to this equation, one obtains $$a^*ap=p=aa^*p.$$
	Also $f\chi_B=0$, so $aq=a(1-p)=0$, and hence $a = ap$. \end{proof}
	
Another such application of the SI property which guarantees the existence of nontrivial projection is given later in Remark \ref{nontrivial_projection}(i).

As an application of Subsection 4.1, Theorem \ref{theorem6.3}, consider the Calkin algebra $B(\mathcal H)/ K(\mathcal H)$ which is a unital $C^*$-algebra and the quotient map $\pi : B(\mathcal H) \rightarrow B(\mathcal H)/K(\mathcal H)$. For $T \in B(\mathcal H)$, $\sigma_e(T)$ is called the essential spectrum of $T$ which is defined as the spectrum of $\pi(T)$ in the Calkin algebra, i.e., $\sigma(\pi(T)) :=\sigma_e(T)$. We found that the SI property of $\Sc(T, T^*)$ generated by an essentially normal operator determines the spectral thinness of the essential spectrum of $T$ in Corollary \ref{corollary6.4} below. And given the SI property for $\Sc(T, T^*)$, we also obtain 
in Corollary \ref{corollary7.3} a complete characterization of which subnormal operators $T$ are essentially normal (i.e., when $\pi(T)$ is  normal in the Calkin algebra) in terms of the area of the essential spectrum. 
 
\begin{corollary}\label{corollary6.4}
	For $T\in B(\mathcal{H})$ an essentially normal operator, if $\Sc(T,T^*)$ is an SI semigroup, then either 
$$\sigma_e(T)\subset \mathbb{R}~\mbox{\quad or \quad}~\sigma_e(T)\subset S^1\cup \{0\}.$$
\end{corollary}
\begin{proof}
	Suppose $\Sc(T,T^*)$ is an SI semigroup. Then, using the $*$-homomorphism 
 $\pi$, note that the semigroup $\Sc(\pi(T),\pi(T)^*)$ is also an SI semigroup in $B(\mathcal{H})/\mathcal{K}(\mathcal{H})$). That $\pi$ preserves the SI property for singly generated semigroups is a straightforward computation using the easily proved fact that  $\Sc(T,T^*)$ is an SI semigroup if and only if for each $A \in  \Sc(T,T^*)$, $A^* = XAY$ for some $X,Y \in \Sc(T,T^*) \cup \{I\}$.
Moreover, normality of $\pi(T)$ implies one of the two possibilities: either $\pi(T)$ is selfadjoint or $\pi(T)$ is a nonselfadjoint normal element. In the former case, $\sigma(\pi(T)) = \sigma_e(T)\subset \mathbb{R}$, and in the latter case, by Theorem~\ref{theorem6.3}, 
$\sigma_e(T) = \sigma(\pi(T))\subset S^1\cup\{0\}.$
\end{proof}
It follows from \cite[Corollary 31.15]{Conway} that for a subnormal operator $T$, if $\area (\sigma_e(T))=0$, then $T$ is essentially normal. The converse does not hold in general (see example in Remark \ref{subnormal} below). If we assume that $\Sc(T, T^*)$ is an SI semigroup, then the converse holds for a subnormal operator $T$, i.e., if $T$  is essentially normal, then $\area (\sigma_e(T))=0$. Indeed, by Corollary \ref{corollary6.4}, $\mbox{either}~\sigma_e(T)\subset \mathbb{R}~\mbox{or}~\sigma_e(T)\subset S^1\cup \{0\}.$ Therefore, $\area(\sigma_e(T))=0$. We summarize these results in the corollary below.

\begin{corollary}\label{corollary7.3}
			Let $T\in B(\mathcal{H})$ be a subnormal operator. Suppose $\Sc(T,T^*)$ is an SI semigroup. Then,
	$$T  ~\mbox{is essentially normal if and only if } \area (\sigma_e(T))=0.$$ 
	\end{corollary}

 \begin{remark}\label{subnormal}
 The conclusion in Corollary~\ref{corollary7.3} does not hold if we drop the hypothesis that $\Sc(T,T^*)$ is an SI semigroup. For instance, the multiplication operator $M_z$ considered in Example \ref{ee-1} is normal and so is essentially normal. And from there recall that $\Sc(M_z, M_z)$ is not SI. It follows from \cite[Chapter XI, Section 4, Proposition 4.6]{Con90} that for a normal operator $N$, $$\sigma(N)\setminus \sigma_{e}(N) = \{\lambda \in \sigma(N) : \lambda \text{ is an isolated point of $\sigma(N)$ that is an eigenvalue of finite multiplicity}\}.$$ Since $\sigma(M_z)$ has no isolated point (as  $\sigma(M_z)=\overline{\mathbb{D}}$, see Example \ref{ee-1}),  $\sigma(M_z)\setminus\sigma_e(M_z)=\emptyset$. Hence, $\sigma_e(M_z)=\sigma(M_z)$. 
This implies that $\area(\sigma_e(M_z))\neq 0.$
  \end{remark}
 
  For an essentially normal operator, the SI property of $\Sc(T, T^*)$  bears on the thinness of the essential spectrum in that the area of the essential spectrum must be zero. But the full spectrum of an essentially normal operator under the SI property need not be thin, in fact, the spectrum could be the closed unit disc as proved in Theorem \ref{fullspec} below. 

We recall here a few definitions that are used in the proof of Theorem \ref{fullspec}. For $A \in B(\mathcal H)$, the left essential spectrum of $A$ and the right essential spectrum of $A$ are defined as $\sigma_{le}(A) = \sigma_{l}(\pi(A))$ and $\sigma_{re}(A) = \sigma_{r}(\pi(A))$ respectively, where $\sigma_{l}(\pi(A))$ and $\sigma_{r}(\pi(A))$ denote the left and the right spectrum of $\pi(A)$ respectively (see \cite[Chapter XI, Definition 4.1]{Con90}).
  
\begin{theorem}\label{fullspec}
	Let $W$ be a  weighted shift with all nonzero complex weights $\{\alpha_n\}$. If $W$ is essentially normal and $\Sc(W,W^*)$ is SI, then $\sigma_e(W)\subseteq S^1\cup \{0\}$ and 
$\sigma(W)=\overline{\mathbb{D}}$. Moreover, $\liminf_n(|\alpha_1\alpha_2...\alpha_n|)^{1/n}=1$.
\end{theorem}
\begin{proof}
	Since $\Sc(W,W^*)$ is SI and $W$ is an infinite-rank nonselfadjoint operator (as $\alpha_n\ne 0$ for $n \geq 1$), it follows by contrapositive from \cite[Theorem 1.17]{PW21} that $W\notin K(\mathcal{H})$. Furthermore, by Corollary \ref{corollary6.4}, the SI property of $\Sc(W, W^*)$ for an essentially normal nonselfadjoint operator $W$ implies that either $\sigma_e(W) \subset \mathbb R$ or $\sigma_e(W)\subseteq S^1\cup \{0\}$. But note that $\sigma_e(W) \not \subset \mathbb R$. Indeed if $\sigma_e(W) \subset \mathbb R$ then $\pi(W)$ is selfadjoint in $B(\mathcal{H})/\mathcal{K}(\mathcal{H})$. Hence, $\pi(W) = \pi(W^*)$ which is equivalent to $W-W^* \in K(\mathcal H)$ which, since $W$ is a weighted shift, further implies that $W \in K(\mathcal H)$, contradicting the non-compactness of $W$. Therefore, $\sigma_e(W) \not \subset \mathbb R$.

 We next prove that $\sigma(W)=\overline{\mathbb{D}}$. Firstly one has that \color{black} $\sigma(W)\neq \{0\}$. Indeed, if $\sigma(W) = \{0\}$, then $\sigma_e(W) = \{0\}$ as $\sigma_e(W)\subseteq \sigma(W)=\{0\}$. Moreover, when $W$ is essentially normal, $\pi(W)$ is normal in $B(\mathcal H)/K(\mathcal H)$ and hence the spectral radius of $\pi(W)$ is equal to its norm (\cite[Lemma 1.2.7]{Lin}). The spectral radius of $\pi(W)$ is equal to 
$\text{max}\{|\lambda| \mid \lambda \in \sigma_e(W) = \sigma(\pi(W)) \}$. Since  $\sigma_e(W) = 0$, the norm of $\pi(W)$ is equal to zero which further implies that $\pi(W) = 0$, or equivalently, $W\in K(\mathcal{H})$, contradicting $W\notin K(\mathcal{H})$. Therefore, $\sigma(W)\neq \{0\}$.
Then since weighted shifts have spectra that must be closed disks centered at the origin \cite[Corollary (Kelley)]{Ridge} and this spectrum is nonzero, it must be a nonzero closed disk with center at the origin, which in fact has no isolated points. From this and \cite[Chapter XI, Proposition 4.2(a) and Theorem 6.8]{Con90} it follows that $\partial(\sigma(W))\subseteq \sigma_{e}(W)$. And as we have proven $\sigma_e(W)\subseteq S^1\cup \{0\}$, one has $\partial(\sigma(W))\subseteq S^1\cup\{0\}$.
Then because the disk $\sigma(W)\neq \{0\}$ and has boundary in $ S^1\cup\{0\}$, 
 that disk must be $\sigma(W)=\overline{\mathbb{D}}$.

We next prove that $\liminf_n(|\alpha_1\alpha_2...\alpha_n|)^{1/n}=1$. For $A \in B(\mathcal{H})$, 
$\sigma_{ap}(A)=\sigma_{le}(A)\cup\{\lambda\in\sigma_p(A) \mid \text{dim\,} \ker(A-\lambda)<\infty\}$ where $\sigma_{ap}(A)$ denotes the approximate point spectrum of $A$ and $\sigma_{le}(A)$ denotes the left-essential spectrum of $A$ (see \cite[Chapter XI, Proposition 4.4]{Con90}). In the case of a weighted shift $W$, the point spectrum $\sigma_p(W)=\phi$ as $\alpha_n \ne 0$ for $n\geq 1$ \cite[Solution 93]{Hal82} (also easy direct computation). Therefore, $\sigma_{ap}(W)=\sigma_{le}(W)\subseteq\sigma_e(W)$. Since $\sigma_e(W)\subseteq S^1\cup\{0\}$, one has $\sigma_{ap}(W)\subseteq S^1\cup\{0\}$. Since $\sigma(W)=\overline{\mathbb{D}}$ and so the boundary of the spectrum 
$S^1 = \partial\sigma(W) \subset \sigma_{ap}(W)$ \cite[Problem 78]{Hal82}, therefore $S^1\subseteq\sigma_{ap}(W)\subseteq S^1\cup\{0\}$.  
Also we have the disjoint union $\sigma_{ap}(W)\cup\Gamma (W) = \sigma(W)=\overline{\mathbb{D}}$ where $\Gamma(W)$ is the residual spectrum of $W$. Since $\sigma_{ap}(W) \cap \Gamma (W) = \emptyset$, so $\mathbb{D}\setminus\{0\}\subseteq \Gamma (W)$.  Using the fact that  for a  weighted shift with nonzero weights $\{\alpha_n\}$, $\sigma_p({W^*})$ is a disc with center $0$ and radius $\liminf_n(|\alpha_1\alpha_2...\alpha_n|)^{1/n}$ (see \cite[Solution 93]{Hal82}), 
and that for any  weighted shift with weights $\{\alpha_n\}$, $\sigma_p(W^*)=\Gamma (W)^*$ \cite[Solution 73]{Hal82},
one has $\mathbb{D}\setminus\{0\}\subseteq \sigma_p(W^*)$. 
And then from  $\mathbb{D}\setminus\{0\}\subseteq \sigma_p({W^*}) \subseteq \sigma(W^*)=\overline{\sigma(W)}=\overline{\mathbb{D}}$, we obtain $\liminf_n(|\alpha_1\alpha_2...\alpha_n|)^{1/n}=1$. \end{proof}
	\begin{corollary}
	Let $W$ be a weighted shift with the weight sequence $\{\alpha_n\}$ of nonzero numbers and $|\alpha_n|\longrightarrow \alpha$ for some $\alpha \in \mathbb{R}^+\cup\{0\}$as $n \rightarrow \infty$. If $\Sc(W,W^*)$ is SI, then $\alpha=1$. 
\end{corollary}
\begin{proof}
	Since$\{|\alpha_n|\}$ is convergent, an elementary computation shows that $W$ is essentially normal. Therefore, by Theorem \ref{fullspec}, $\sigma(W)=\overline{\mathbb{D}}$. On the other hand, after verifying $W$ is injective, $|\alpha_n|\longrightarrow \alpha$ implies that $\sigma(W)=\{z\in \mathbb{C}\mid |z|\leq \alpha\}$ \cite[Corollary 1]{Ridge}.  Therefore $\alpha=1$.
\end{proof}

As promised earlier, we next characterize the SI semigroup $\Sc(a,a^*)$ generated by a non-normal idempotent element $a$ in a $C^*$-algebra (Theorem \ref{idempotent}). And as a byproduct, we obtain nontrivial projections in $C^*(a)$ under the SI assumption on $\Sc(a, a^*)$ (Remark \ref{nontrivial_projection}). We first recall the definition of a partial isometry for an abstract $C^*$-algebra.
\begin{definition}\cite[Definition 5.1.4]{Olsen}\label{definition of pi}
	An element $a$ in a $C^*$-algebra $\mathcal A$ is called a partial isometry when $a^*a$ is a projection.
\end{definition}
\noindent The familiar equivalent statements about partial isometry in $B(\mathcal H)$ also hold for any $C^*$-algebra, that is, $a$ is a partial isometry if and only if $a^*a$ is a projection if and only if $a=aa^*a$ (see \cite[Exercise 5.A(d)]{Olsen}).

If $a$ is the unit element of $\mathcal A$, then $\Sc(a, a^*) = C^*(a) = \{1\}$, and so there is no nontrivial projection. So, we assume that $a \neq 1$ henceforth in this discussion.
Note that if $a$ is a nonzero selfadjoint idempotent, i.e., $a^2=a=a^*$, then $a$ itself is a projection which is clearly a partial isometry. And since $a$ is selfadjoint, $\Sc(a,a^*)$ is automatically SI. So in this case, $C^*(a)$ has a nontrivial projection, namely $a$ itself (as $0 \neq a \neq 1$). 

If $a$ is normal nonselfadjoint and idempotent, the relations $a^*a = aa^*$ and $a^2 = a$ imply that $a^*a$ is a projection and so also $aa^*$ is a projection. Moreover, if $a^*a= 1$ then multiplying both sides with $a$ and using $a= a^2$ implies selfadjointness of $a$ contradicting the nonselfadjointness of $a$. So, $0 \neq a^*a \neq 1$ which implies that $a^*a$ is a nontrivial projection.  Observe that the SI property does not play any role in the existence of nontrivial projections in $C^*(a)$ when generated by a normal nonselfadjoint idempotent. 
So the interesting case for us is when $a$ is a non-normal idempotent which is addressed in the following theorem.

\begin{theorem}\label{idempotent}
	For a non-normal idempotent element $a\in \mathcal{A}$ in a unital $C^*$-algebra, the following are equivalent.
	\begin{enumerate}[label=(\roman*)]
		\item $\Sc(a,a^*)$ is an SI semigroup.
		\item $a$ is a partial isometry.
		\item $\Sc(a,a^*)$ is simple.
	\end{enumerate}
\end{theorem}
\begin{proof}
	Since $a$ is a non-normal idempotent, $a$ is not equal to the unit element in $\mathcal{A}$ and the semigroup list for $\Sc(a,a^*)$ described in the introduction of this subsection prior to Theorem~\ref{theorem6.3}) reduces to the following list:
	\begin{equation}\label{semigroup_list}
	\Sc(a,a^*)= \{a,a^*,(aa^*)^k,(aa^*)^ka,(a^*a)^k,(a^*a)^ka^* : k \geq 1\}.
	\end{equation}
	 (i) $\Rightarrow$ (ii): Suppose $\Sc(a,a^*)$ is an SI semigroup. To show that $a$ is a partial isometry, we prove that $a^*a$ is a projection. Indeed, the SI property of $\Sc(a,a^*)$ implies that the principal ideal $(a)_{\Sc(a,a^*)}$ is selfadjoint. Therefore, $a^*=xay$ for some $x,y\in \Sc(a,a^*)\cup\{1\}$; where both $x,y$ are not equal to $1$ because $a$ is nonselfadjoint. Since $a$ is not normal, it is not selfadjoint and $a^*=xay$ implies that $xay$ is not selfadjoint and so $xay$ cannot be in first, third and fifth form in display~(\ref{semigroup_list}). Moreover, since $xay$ contains an $a$ in it,  $xay$ must be either in fourth or sixth form. If $xay$ has fourth form. Then $xay=(aa^*)^ka$ for some $k\geq 1$. Since $a^*=xay$ so $a^*=(aa^*)^ka$. Multiplying by $a^*$ on both sides we obtain, ${a^*}^2=(aa^*)^{k+1}$ which implies that $a^*=(aa^*)^{k+1}$, contradicting the nonselfadjointness of $a^*$. Therefore, $xay$ must be in sixth form, i.e., $xay=(a^*a)^ka^*$ for some $k\ge 1$. Then using $a^*=xay$ we have $a^*=(a^*a)^ka^*$. We then multiply by $a$ on both sides which further implies that $a^*a=(a^*a)^{k+1}$. Since $a^*a$ is normal, by \cite[Theorem 2.1.13]{Murphy}, it follows that $\sigma(a^*a) \subset \{0, 1\}$ and that $a^*a$ is idempotent. So $a^*a$ is a projection.
	
(ii) $\Rightarrow$ (iii): Suppose $a$ is a partial isometry. Then $a^*a$ is a projection and $a=aa^*a$. Therefore, the semigroup list in~(\ref{semigroup_list}) reduces to the following list:
	$$\Sc(a,a^*)=	\{a,a^*,aa^*,a^*a\}.$$
	Using the relation $a=aa^*a$, one can easily check that $\Sc(a,a^*)$ is simple.
	
(iii) $\Rightarrow$ (i):	 Simple semigroups are automatically SI.
\end{proof}
An immediate consequence under the hypothesis of the above theorem is highlighted in the remark below.
\begin{remark}\label{nontrivial_projection}
(i) For a non-normal idempotent, in the proof of Theorem \ref{idempotent}(i)$\Rightarrow$(ii), we showed that if $\Sc(a, a^*)$ is SI, then $\sigma(a^*a) \subset \{0, 1\}$ and $a^*a$ is a projection. We further conclude the existence of a nontrivial projection in $C^*(a)$. Indeed, suppose $\sigma(a^*a)$ is a singleton set, namely, $\sigma(a^*a) = \{1\}$ (since $\sigma(a^*a) = 0$ implies $a = 0$ contradicting non-normality of $a$). Since $\sigma(a^*a) = \{1\}$, $a^*a$ is invertible.  Also $a^*a = (a^*a)^2$, so multiplying $(a^*a)^{-1}$ on both sides, we obtain $1 = a^*a$. Multiplying by $a$ on both sides to $1 = a^*a$ and using $a = a^2$, we obtain $a= a^*a$ which contradicts the non-normality of $a$. Therefore, $\sigma(a^*a) = \{0, 1\}$ which is disconnected. Hence, $C^*(a^*a)$ has nontrivial projections. Since $C^*(a^*a) \subset C^*(a)$, it follows that $C^*(a)$ has nontrivial projections.

(ii)  We summarize the SI characterization of $\Sc(a, a^*)$ generated by an idempotent. For a normal idempotent, $\Sc(a, a^*)$ is automatically SI. Indeed, as discussed in the paragraph preceding Theorem \ref{idempotent}, $a^*a$ is a projection and so $a$ is a partial isometry. It follows from Theorem \ref{idempotent}(ii)$\Rightarrow$(iii) that $\Sc(a, a^*)$ is simple, and hence SI. 
 For a non-normal idempotent, $\Sc(a, a^*)$ is SI if and only if $a$ is a partial isometry if and only if $\Sc(a, a^*)$ is simple (Theorem \ref{idempotent}).
	\end{remark}

Based on the evidences found so far on the impact of the SI property of $\Sc(T, T^*)$ on the spectrum of $T$ for special classes of operators, for instance, in Theorem \ref{theorem 3.1}, Remark \ref{almost p-spectrum}, Remark \ref{R6.1}, Theorem \ref{theorem6.3}, and Theorem \ref{fullspec}, we have $\sigma(T) \subset \overline{\mathbb D}$ for $\Sc(T, T^*)$ semigroup to be SI. 
 
\noindent \textit{Conjecture: If $\Sc(T, T^*)$ is an SI semigroup generated by a nonselfadjoint operator $T$, then the spectrum of $T$ is a subset of the closed unit disk.}

\section*{Acknowledgement}
The first author would like to thank the Science and Engineering Research Board, Core Research Grant 002514, and the third author would like to thank Simons Foundation collaboration grants 245014 and 636554 for undertaking this research work.

\end{document}